\documentclass[a4paper,11pt]{amsart}

\usepackage{amsmath}
\usepackage{amssymb}
\usepackage{amsthm}
\usepackage{verbatim}
\usepackage{tikz}
\usepackage[all]{xy}
\usepackage{bm}
\usepackage{enumerate}

\usepackage{color}

\definecolor{newcolorchange}{rgb}{0.5, 0.2, 0.2}
	\definecolor{airforceblue}{rgb}{0.1, 0.1, 0.5}

\newtheorem{thm}{Theorem}[section]
\newtheorem{prop}[thm]{Proposition}
\newtheorem{cor}[thm]{Corollary}
\newtheorem{lem}[thm]{Lemma}

\theoremstyle{definition}

\newtheorem{dfn}[thm]{Definition}

\newtheorem{example}[thm]{Example}
\newtheorem{rmk}[thm]{Remark}

\numberwithin{equation}{section}
\newcommand{\cM}{M}

\newcommand{\bG}{\mathbb{G}}

\newcommand{\id}{\textrm{id}}

\newcommand{\N}{\mathbb{N}}

\newcommand{\C}{\mathbb{C}}

\newcommand{\cL}{\mathcal{L}}

\newcommand{\minotimes}{\otimes_{{\rm min}}}

\newcommand{\barotimes}{\overline{\otimes}}

\newcommand{\Ad}{{\rm Ad}}

\newcommand{\cA}{\mathcal{A}}

\newcommand{\op}{{\rm op}}

\newcommand{\RH}{H_{\mathbb{R}}}

\newcommand{\varphiO}{\varphi_{\Omega}}

\newcommand{\dimq}{{\rm qdim}}

\newcommand{\cS}{\mathcal{S}}
\newcommand{\notimes}{\otimes_\nabla}

\newcommand{\algtimes}{\otimes_{{\rm alg}}}

\newcommand{\Hom}{ {\rm Hom} }

\newcommand{\cF}{\mathcal{F}}

\title[$L_2$-cohomology, derivations and semi-groups on $q$-Gaussian algebras]{$L_2$-cohomology, derivations and  quantum Markov semi-groups on $q$-Gaussian algebras}

\date{\noindent \today.  \\
 }

\author[Martijn Caspers]{Martijn Caspers}

\address{TU Delft, EWI/DIAM,
	P.O.Box 5031,
	2600 GA Delft,
	The Netherlands}
\email{m.p.t.caspers@tudelft.nl}

\author{Yusuke Isono}
\address{Research Institute for Mathematical Sciences, Kyoto University,
	606-8502,
	Kyoto, Japan
}

\email{isono@kurims.kyoto-u.ac.jp}

\author{Mateusz Wasilewski}
\address{KU Leuven,
Celestijnenlaan 200b - bus 2400,
3001 Leuven
}
\email{mateusz.wasilewski@kuleuven.be}

\begin{document}
\begin{abstract}
  	We study (quasi-)cohomological properties through an analysis of quantum Markov semi-groups. We construct higher order Hochschild cocycles using gradient forms associated with a quantum Markov semi-group. By using Schatten-$\cS_p$ estimates we analyze when these cocycles take values in the coarse bimodule. For the 1-cocycles (the derivations) we show that under natural conditions we obtain the Akemann-Ostrand property. We apply this  to $q$-Gaussian algebras $\Gamma_q(\RH)$. As a result $q$-Gaussians satisfy AO$^+$ for $\vert q \vert \leqslant \dim(\RH)^{-1/2}$. This includes a new range of $q$ in low dimensions compared to Shlyakhtenko \cite{Shlyakhtenko}.
  \end{abstract}
\maketitle

  \section{Introduction}\label{Sect=Intro}
  The aim of this paper is to connect the theory of quantum Markov semi-groups to certain cohomological properties of their  algebras. Quantum Markov semi-groups are continuous semi-groups of trace preserving unital completely positive maps on a finite von Neumann algebra. Such quantum Markov semi-groups naturally arise in time evolutions of open systems that undergo decoherence. With the emergence of non-commutative probability, their theory had been investigated thoroughly. We refer in particular to the papers of Goldstein-Lindsay \cite{GoldsteinLindsay} and Cipriani \cite{Cipriani}.

  \vspace{0.3cm}

  A result that is of fundamental importance to us was obtained by Cipriani and Sauvageot in \cite{CiprianiSauvageot}. They showed that any generator $\Delta$ of a quantum Markov semi-group $(e^{-t \Delta})_{t \geq 0}$ admits a closable derivation $\partial$ as its square root, i.e. $\Delta = \partial^\ast \overline{\partial}$. Since derivations are 1-cocycles in Hochschild cohomology, this is the first instance that shows relevance of quantum Markov semi-groups to cohomology.

  The results from \cite{CiprianiSauvageot} had several consequences in non-commutative potential theory and quantum probability. Much more recently also links to approximation and rigidity properties of von Neumann algebras were made. In particular, amenability  \cite{CiprianiSauvageotAdvances} and Haagerup property \cite{CaspersSkalskiCMP} can be characterized in terms of quantum Markov semi-groups with sufficient decay (see also \cite[Appendix]{CaspersGradient}). Further rigidity properties of von Neumann algebras can be obtained \cite{CaspersGradient} by using quantum Markov semi-groups as input for the machinery developed by Ozawa-Popa \cite{OzawaPopaAnnals}, \cite{OzawaPopaAJM} and Peterson \cite{Peterson}. A crucial tool here is the gradient form (or carr\'e du champ) of a quantum Markov semi-group (or classically a diffusion semi-group)

  \vspace{0.3cm}

  We describe these rigidity results a bit further. In the celebrated paper  \cite{Voiculescu} Voiculescu showed, using free entropy, that free group factors do not possess a Cartan subalgebra. Later on, using completely different methods, Ozawa and Popa \cite{OzawaPopaAnnals} re-obtained this result. Ozawa and Popa in fact prove a stronger result. Namely, they show that the normalizer of any amenable diffuse von Neumann subalgebra generates a von Neumann algebra that is amenable again. This property became known as \emph{strong solidity} and plays a central role in the theory. Strong solidity results are usually proved from a combination of weak compactness and a malleable deformation or a group geometric/cohomological property.

In particular, the rigidity results of \cite{OzawaPopaAJM}, \cite{Peterson} show that proper 1-cocycles in group cohomology with values in the coarse bimodule (or any bimodule weakly contained in it)  can be used to show that von Neumann algebras are strongly solid.  After \cite{OzawaPopaAJM}, \cite{Peterson} these results were improved upon in \cite{ChifanSinclair}, \cite{PopaVaesCrelle} and \cite{IsonoTransactions}, where it was shown that also the Akemann-Ostrand property, which in the case of a group von Neumann algebra compares to quasi-cocycles in bounded (group) cohomology, can be used to get strong solidity results and further prime factorization results \cite{OzawaPopaPrime}.

  \vspace{0.3cm}

  In the current paper we address the question whether also derivations in Hochschild cohomology can be used to obtain (quasi-)cohomological properties like the Akemann-Ostrand property. We do this as follows. Fix a finite von Neumann algebra $M$ with a quantum Markov semi-group on it having a dense $\ast$-algebra $\cA$ in the domain of its generator. We first show that we can construct natural $n$-cocycles in the Hochschild cohomology of $\cA$. The case $n=1$ reduces to the work of Cipriani-Sauvageot mentioned above \cite{CiprianiSauvageot}. The $n$-cocycles are defined inductively and in each inductive step their coefficient bimodule changes by a construction which we refer to as the `gradient tensor product'. Say that the $n$-cocycles take values in the $n$-fold gradient tensor $L_2(M)_{\nabla^{(n)}}$ (we omit the construction here in the introduction). We then analyze when $L_2(M)_{\nabla^{(n)}}$ is quasi-contained in the coarse bimodule; meaning that it is contained as a bimodule in an infinite direct sum of copies of the coarse bimodule. In order to do this we introduce the notion of gradient-$\cS_p$ for a quantum Markov semi-group and prove the following.

  \vspace{0.3cm}

  \noindent  {\bf Theorem A.} If a quantum Markov semi-group is gradient-$\cS_{2n}$ then $L_2(M)_{\nabla^{(n)}}$ is quasi-contained in the coarse bimodule of $M$.

  \vspace{0.3cm}

  So essentially under gradient-$\cS_{2n}$ the cohomology takes place in the coefficient bimodule given by the coarse bimodule. Note that gradient-$\cS_{2n}$ becomes weaker for higher $n$. We illustrate this for $q$-Gaussians algebras introduced by Bo\.{z}ejko and Speicher \cite{BozejkoSpeicher}, see also \cite{BKS}.

  \vspace{0.3cm}

\noindent  {\bf Theorem B.}  For   $\vert q \vert < \dim(\RH)^{-1/p}$  we have that the Ornstein-Uhlenbeck semi-group on the $q$-Gaussian algebra $\Gamma_q(\RH)$ is gradient-$\cS_{p}$.

  \vspace{0.3cm}

  The importance of Theorems A and B is so far mostly witnessed in the case $n=1$. In fact, in Theorem \ref{CartanEilenberg} we show that the cohomology of the $\ast$-algebras associated with $q$-Gaussians vanishes for $n \geq 2$ so that the cocycles we construct are in fact coboundaries. For $n = 1$ we show that Theorem B can be used to obtain further results that serve as in input for the machinery developed in   \cite{PopaVaesCrelle} and \cite{IsonoTransactions}.  We give sufficient conditions on a derivation to imply the Akemann-Ostrand property. We analyze these conditions in the case of group algebras but also many other algebras by assuming that the quantum Markov semi-group  has a type of filtration (or is radial in some sense). In many known natural examples these condtions are verified, see the end of Section \ref{Sect=AO}. As a culminating result we single out the following.

    \vspace{0.3cm}

  \noindent  {\bf Theorem C.} For $\vert q \vert \leqslant \dim(\RH)^{-1/2}$  the $q$-Gaussian algebra $\Gamma_q(\RH)$ satisfies the Akemann-Ostrand property; more precisely condtion AO$^+$ from \cite{IsonoTransactions}.

  \vspace{0.3cm}

  In \cite{Shlyakhtenko} Shlyakhtenko obtained the same result for $\vert q \vert < \sqrt{2} -1$ using that the $q$-Toeplitz algebra is nuclear in this range. In fact the $q$-Toeplitz algebras are isomorphic within this range \cite{NicaDykema} and nuclearity for $\sqrt{2} -1 \leqslant \vert q \vert <1$  is still an open problem. Theorem C thus strictly extends the known range of $q$ for which $\Gamma_q(\RH)$ has AO$^+$ for dimension up until 5.

  \vspace{0.3cm}

  The outline of this paper is as follows. After the preliminaries we introduce the gradient tensor construction in Section \ref{Sect=GradientTensor}. We analyze when a repeated tensor product is contained in the coarse bimodule through the notion of being gradient-$\cS_p$ (Theorem A). We also show that in the group case, under a domain condition, one may always average quantum Markov semi-groups to a semi-group of Fourier multipliers retaining the gradient $\cS_p$-properties.  In Section \ref{Sect=Orhnstein-Uhlenbeck} we illustrate this using $q$-Gaussians and we show that the Ornstein-Uhlenbeck semi-group is gradient-$\cS_{p}$ for $\vert q \vert < \dim(\RH)^{-1/p}$ (Theorem B). Section \ref{Sect=AO} is then concerned with the question of which derivations imply condition AO$^+$. We give sufficient conditions for the Cipriani-Sauvageot derivations \cite{CiprianiSauvageot}; amongst many other examples this includes $q$-Gaussians and we conclude Theorem C.

  \vspace{0.3cm}

  \noindent {\bf Acknowledgements.} The authors thank    Ebrahim Samei for useful discussions and Adam Skalski  for sharing a short proof of Lemma \ref{Lem=Average}.  The authors thank the referees for several comments that led to an improvement of the manuscript.
  Research of MW was partially supported by  National Science Centre (NCN) grant 2016/21/N/ST1/02499, long term structural funding -- Methusalem grant of the Flemish Government -- and by European Research Council Consolidator Grant 614195 RIGIDITY. Parts of this work were completed during a visit of MW to TU Delft; he thanks the university and the host for a stimulating research environment.

  \section{Preliminaries}

  \subsection{General conventions and notation}
  Throughout the paper $M$ is a finite von Neumann algebra with a trace $\tau$.  Let $\Omega_\tau \in L_2(M) := L_2(M, \tau)$ be the cyclic vector given by  $1 \in M$.
  We denote by $\cS_p$ the Schatten-von Neumann class, which is the non-commutative $L_p$-space associated with $B(H)$ and its trace.

  \subsection{Containment of bimodules}
  Let $\cA$ be a $\ast$-algebra. By an $\cA$--$\cA$ bimodule $H$ we mean a Hilbert space together with $\ast$-homomorphisms $\pi_l: \cA \rightarrow B(H)$ and $\pi_r: \cA^{op} \rightarrow B(H)$ whose images commute.  If $\cA$ is a von Neumann algebra then we assume moreover that $\pi_l$ and $\pi_r$ are normal.

   We write $L_2(M)$ for the trivial bimodule and $L_2(M) \otimes L_2(M)$ for the coarse bimodule with the usual left and right actions.
 Let $H$ and $K$ be $M$-$M$ bimodules. We say that $H$ is {\bf contained} in $K$ if $H$ is (isomorphic to) a sub-bimodule of $K$; i.e. $H$ is a closed subspace of $K$ that is invariant for the bimodule actions of $M$. We say that $H$ is {\bf quasi-contained} in $K$ if $H$ is contained in $\oplus_{i \in I} H$ for some index set $I$.  We use Popa's definition of weak containment \cite{PopaIncrest}.

\begin{dfn}
	Let $M$ be a von Neumann algebra and let $H$ and $K$ be $M$-$M$-bimodules. We say that $H$ is {\bf weakly contained} in $K$ if for every $\epsilon > 0$, every finite set $F \subseteq M$ and every $\xi \in H$ there exists  finitely many $\eta_j \in K$ indexed by $j \in G$ such that for $x,y \in F$,
	\[
	\vert \langle x \xi y , \xi \rangle - \sum_{j \in G} \langle x \eta_j y, \eta_j \rangle \vert < \varepsilon.
	\]
	Notation $H \preceq K$.
\end{dfn}

The following clarifies the connection to containment in the way we encounter it in this paper.

\begin{lem}\label{Lem=Containment}
	Let $M$ be a von Neumann algebra and $\cA$ a $\sigma$-weakly dense $\ast$-subalgebra of $M$ with norm closure $A$. Let $H$ be an $A$-$A$-bimodule and let $K$ be an $M$-$M$-bimdoule.   Suppose that there exists a dense subspace $D \subseteq H$ such that for every $\xi \in D$ there exists an $\eta \in K$ such that for every $x,y \in \cA$ we have
	\[
	\langle x \xi y, \xi \rangle = \langle x \eta y, \eta \rangle.
	\]
	Then for every $\xi \in D$ the sub-bimodule $H_\xi := \overline{ A \xi A}$ of $H$ is contained in $K$   as $A$-$A$ bimodules. Consequently,  $H$  is  contained in a (possibly infinite) direct sum of copies of $K$ and further the $A$-$A$ bimodule actions on $H$ extend to normal $M$-$M$-bimodule actions.
\end{lem}
\begin{proof}
	Take $\xi \in D$. By assumption there exists $\eta \in K$ such that for all $x,y \in \cA$ we have that $\langle x \xi y, \xi \rangle = \langle x \eta y, \eta \rangle$. Then set $U: H_\xi \rightarrow K$ by $a \xi b \mapsto a \eta b$ where $a,b \in \cA$  and it follows that this map is isometric and intertwines the bimodule actions (of $\cA$ and then by continuity of $A$).

Since $UH_\xi$ admits normal $M$-$M$ bimodule actions, we can extend the $A$-$A$ actions on $H_\xi$ to normal $M$-$M$ bimodule actions.

	Let $\Sigma$ be the set of all families of unit vectors $(\xi_i)_i$ in $H$ such that each $H_{\xi_i}:=\overline{A \xi_i A}$ is a sub $A$-$A$ bimodule of $K$ and that all $H_{\xi_i}$ are orthogonal with each other. By Zorn's lemma, take a maximal element $(\xi_i)_i$ in $\Sigma$. Let  $P_i$ be the orthogonal projection onto $H_{\xi_i}$ and suppose by contradiction that $P:=\sum_i P_i \neq 1_H$. Note that $P$ commutes with $A$-$A$ bimodule actions. Let $\xi\in D$ be such that $(1-P)\xi\neq 0$ and fix $\eta\in K$ such that $\langle a\xi b,\xi\rangle=\langle a\eta b,\eta\rangle$.
Then observe that for all finitely many $a_i,b_i \in \mathcal A$,
	$$ \| \sum_i a_i (1-P)\xi b_i \| =   \| (1-P) \sum_i a_i \xi b_i \| \leqslant \| \sum_i a_i \xi b_i \| = \| \sum_i a_i \eta b_i \|.$$
One can define a contraction $v\colon \overline{A \eta A} \to \overline{A (1-P)\xi A}$ by $v (a \eta b ) = a (1-P)\xi b$ for $a,b\in \mathcal A$. Then since $v^*v$ commutes with $A$-$A$ bimodule actions, if we put $\eta':= (v^*v)^{1/2}\eta$, then it satisfies
	$$\langle a (1-P)\xi b, (1-P)\xi\rangle = \langle a \eta' b,\eta'\rangle, \quad \text{for any }a,b \in \mathcal A.$$
This contradicts the maximality of $(\xi_i)_i$.
We conclude that there is an $A$-$A$ bimodule embedding
	$$ H= \bigoplus_i H_{\xi_i} \subset \bigoplus_i K $$
and we can extend the $A$-$A$ actions to normal $M$-$M$ bimodule actions via this embedding.
\end{proof}

\subsection{Hochschild cohomology} Fix an algebra $\cA$. We define Hochschild cohomology (see \cite{CartanEilenberg}, \cite{ThomGafa} and also \cite{ConnesShlyakhtenko}) through the bar resolution as follows.
Let $F_n$ be the space of all linear maps $f\colon \cA^{\otimes n} \rightarrow H$ to a fixed $\cA$-$\cA$ bimodule $H$. For $f \in F_n$ we define $d_n f\colon \cA^{\otimes n + 1} \rightarrow H$ by
\begin{align*}
&(d_n f)(a_1 \otimes a_2 \otimes \ldots \otimes a_{n+1})
 =  a_1 \cdot f( a_2 \otimes \ldots \otimes a_{n+1}) \\
&+ \sum_{k=1}^{n} (-1)^{k}  f( a_1 \otimes \ldots \otimes a_k a_{k+1} \otimes \ldots \otimes a_{n+1})
+ (-1)^{n+1}f(a_1 \otimes \ldots \otimes a_{n}) \cdot a_{n+1}
\end{align*}
with $a_1, \ldots, a_{n+1} \in \cA$.
Further, set for $\xi \in H$ the map $d_0\xi: \cA \rightarrow H$ by $(d_0 \xi)(a) =  a \xi - \xi a, a \in \cA$.
We get a chain called the {\bf bar resolution},
\[
\ldots \leftarrow^{d_4} F_4 \leftarrow^{d_3} F_{3}  \leftarrow^{d_2} F_2 \leftarrow^{d_1} F_1 \leftarrow^{d_0} H,
\]
with $d_{n+1} \circ d_n = 0$. Let $C^n(\cA, H)$ be the kernel of $d_n$, which we call the {\bf Hochschild $n$-cocycles}. Let $B^n(\cA, H)$ be the image of $d_{n-1}$, which we call the {\bf $n$-coboundaries}. By definition $B^0(\cA, H) = \{ 0\}$.  Let
\[
H^n(\cA, H) = C^n(\cA, H)/B^n(\cA, H)
\]
be the {\bf $n$-th Hochschild cohomology group} with coefficients in $H$.

  \section{Gradient tensoring, $n$-cocycles and gradient-$\cS_p$}\label{Sect=GradientTensor}
The aim of this section is to construct $n$-cocycles in Hochschild cohomology using quantum Markov semi-groups.

  \begin{dfn}
  	A {\bf quantum Markov semi-group} $(\Phi_t)_{t \geq 0}$ on $(\cM, \tau)$ is a semi-group of unital completely positive maps $\Phi_t: \cM \rightarrow \cM$ that preserve the trace, i.e. $\tau \circ \Phi_t = \tau$ for all $t \geq 0$. Further we assume that $\Phi_t$ is symmetric $\tau(\Phi_t(x) y) = \tau(x \Phi_t(y)), x,y \in M$. Moreover, we assume that for every $x \in \cM$ the map $t \mapsto \Phi_t(x)$ is continuous for the strong topology of $\cM$.
  \end{dfn}

  \subsection{Gradient tensoring}

  Let $(\Phi_t)_{t \geq 0}$ be a quantum Markov semi-group on $M$. Let
  \[
  \Phi_t^{(2)}: L_2(M) \rightarrow L_2(M): x \Omega_\tau \mapsto \Phi_t(x) \Omega_\tau,
  \]
   be the corresponding semi-group on $L_2(M)$.    Let $\Delta \geq 0$ be its generator, i.e. $\Delta$ is the unbounded positive (self-adjoint) operator on $L_2(M)$ such that
   \[
   \exp(-t \Delta) =  \Phi_t^{(2)}
   \]
    as a semi-group on $L_2(M)$.

     \vspace{0.3cm}

\noindent {\bf Assumption 1.} We shall assume that there exists a $\sigma$-weakly dense $\ast$-subalgebra $\cA$  of $\cM$ such that $\cA \Omega_\tau$ is contained in the domain of $\Delta$ and moreover  $\Delta(\cA \Omega_\tau) \subseteq \cA\Omega_\tau$. We simply write $\Delta(a), a \in \cA$ for the map $\Delta$ on the level of $\cA$. We assume moreover  that for every $a \in \cA$ the map  $t \mapsto \Phi_t(a)$ is norm continuous.
  Let
  \[
  A := \overline{\cA}^{\Vert \: \cdot \: \Vert}
  \]
   be the C$^\ast$-closure of $\cA$.
    In general we cannot guarantee the existence of such an algebra $\cA$, but in many concrete cases it exists.  The assumption should be compared to similar assumptions and remarks made in \cite[Section 5]{CiprianiFagnola}, \cite{JungeMei} or \cite{CaspersGradient}.

    \vspace{0.3cm}

    We introduce the following definition, which is in principle $\cA$-dependent. For $p = 2$ it was introduced in \cite{CaspersGradient}.

  \begin{dfn}\label{Dfn=GradientSp}
   Let $1 \leqslant p \leqslant \infty$. 	We say that a quantum Markov semi-group $(\Phi_t)_{t \geq 0}$  is {\bf immediately gradient-$\cS_p$} if for every $a,b \in \cA$ and every $t > 0$ we have that
  	\begin{equation}\label{Eqn=GradientMap}
  \Psi^{a,b}_t:	x  \mapsto -\frac{1}{2} \Phi_t \left(  \Delta(ax b) + a \Delta(x) b - \Delta(ax) b - a \Delta(xb)   \right),
  	\end{equation}
  	extends to  a bounded map $x \Omega_\tau \mapsto \Psi^{a,b}_t(x) \Omega_\tau$  on $L_2(\cM)$ that is moreover in $\cS_p$. 	 $(\Phi_t)_{t \geq 0}$  is {\bf  gradient-$\cS_p$} if for every $a, b \in \cA$ the map  \eqref{Eqn=GradientMap} is in $\mathcal{S}_p$ for $t = 0$. We  set  $\Psi^{a,b} = \Psi^{a,b}_0$.
  \end{dfn}

\begin{example}
We illustrate Definition \ref{Dfn=GradientSp} by a simple (counter)example. Suppose that $\mathbb{T}$ is the torus seen as the unit circle in $\mathbb{C}$. Let $e_k(z) = z^k, z \in \mathbb{T}$. Let $\mathcal{A}$ be the span of $e_k, k \in \mathbb{Z}$, i.e. the $\ast$-algebra of trigonometric polynomials. Set  $M= L_\infty(\mathbb{T})$.
\begin{itemize}
\item  Let $\Delta$ be the Laplacian given by $\Delta(e_k) = k^2 e_k$ which generates the {\bf heat semi-group}  $(e^{-t \Delta})_{t \geq 0}$. This semi-group is not gradient-$\cS_p$ for any $1 \leq p \leq \infty$. Indeed, we find that
\[
\Psi^{e_l, e_m}(e_k) = -\frac{1}{2} ( (l+k+m)^2 + k^2 - (l+k)^2 - (k+m)^2 )   e_{l + k + m} = - lm \: e_{l + k + m},
\]
and this map is not even compact on $L_2(\mathbb{T})$ unless $lm = 0$.
\item
However, if we  consider the {\bf Poisson semi-group} $(e^{-t \Delta^{\frac{1}{2}}})_{t \geq 0}$ then we find,
\[
\Psi^{e_l, e_m}(e_k) = -\frac{1}{2} ( \vert l+k+m \vert + \vert k \vert - \vert l+k \vert - \vert k+m \vert )   e_{l + k + m},
\]
which is 0 as soon as $\vert k\vert \geq \vert l \vert + \vert m \vert$. Therefore $\Psi^{e_l, e_m}$ is finite rank and therefore  $(e^{-t \Delta^{\frac{1}{2}}})_{t \geq 0}$ is gradient-$\cS_p$ for every $p \in [1, \infty]$.

\end{itemize}
\end{example}

\begin{rmk}
  Morally, one can expect that a quantum Markov semi-group is gradient-$\cS_p$ if the eigenvalues of the generator grow (almost) linearly (like for the Poisson semi-group) and if there is not too much interaction between the operators $a$ and $b$ in an expression $axb$ if $x$ has a very large `length' (this obviously requires more structure to explain in more detail). The current paper gives  quantitative examples of this, see Section \ref{Sect=Orhnstein-Uhlenbeck} and Theorem B of Section \ref{Sect=Intro}. More examples as well as stability properties for gradient-$\cS_2$ for free products were given in \cite{CaspersGradient}.
\end{rmk}

  In \cite{CiprianiSauvageot} the following bimodule is constructed.  We set the {\bf gradient form} (or  {\bf carr\'e du champ}),
  \[
\Gamma(x,y) = \frac{1}{2}( \Delta(y)^\ast x + y^\ast  \Delta(x) - \Delta(y^\ast x) ), \qquad x,y \in \mathcal{A},
  \]
  which we view as an $\cA$-valued inner product.  Let $H$ be any $A$-$A$-bimodule.
  Equip the algebraic tensor product $\mathcal{A} \algtimes H$ with the (degenerate) inner product
  \[
  \langle x \otimes \xi, y \otimes \eta \rangle =  \langle \Gamma(x,y) \xi, \eta  \rangle, \qquad x,y \in \mathcal{A}, \xi, \eta \in H.
  \]
  The Hilbert space obtained by quotienting out the degenerate part and taking the completion will be called the {\bf gradient tensor product} which we denote by $H_\nabla$. We will denote by $a \notimes \xi, a \in \cA, \xi \in H$ the class of $a \otimes \xi$ in $H_\nabla$.
   $H_\nabla$  has an $\mathcal{A}$-$\cA$ bimodule structure given by the left action
  \begin{equation}\label{Eqn=LeftAction}
  x \cdot (y \notimes \xi) = xy \notimes \xi - x \notimes y \xi, \qquad x,y \in \mathcal{A}, \xi \in H,
  \end{equation}
  and the right action
  \begin{equation}\label{Eqn=RightAction}
  (y \notimes \xi ) \cdot x = y \notimes \xi x, \qquad x,y \in \mathcal{A}, \xi \in H.
  \end{equation}
  \begin{prop}\label{Prop=Module}
  	The left and right actions defined in \eqref{Eqn=LeftAction} and \eqref{Eqn=RightAction} are well-defined contractive left and right actions of $\cA$ that moreover commute. That is, $H_\nabla$ is an $A$-$A$-bimodule where $A$ is the C$^\ast$-closure of $\cA$.
  \end{prop}
\begin{proof}
The proof in case $H = L_2(M)$ is the trivial $M$-$M$-bimodule is given as  \cite[Lemma 3.5]{CiprianiSauvageot}. But in fact the same proof works for any $A$-$A$-bimodule $H$.  Note that the proof that the right action is contractive is straightforward.
\end{proof}

\begin{dfn}
	We call $(\Phi_t)_{t \geq 0}$ {\bf gradient coarse}  if the $A$-$A$-bimodule actions on $L_2(M)_{\nabla}$ extend to normal $M$-$M$-bimodule actions and further $L_2(M)_{\nabla}$ is weakly contained in the coarse bimodule of $M$.
\end{dfn}

By Proposition \ref{Prop=Module} for an $A$-$A$-bimodule $H$ we may define
 \[
 H_{\nabla^{(n)}} = ((H_{\nabla})_{\nabla} \ldots )_\nabla,
 \]
 for the $n$-fold application of $H \mapsto H_\nabla$. The following connects the map \eqref{Eqn=GradientMap} to the bimodule structure defined above.

 \begin{lem}\label{Lem=GradientComp}
 	Let $H$ be an $A$-$A$-bimodule.
 	For $x, y, a, b \in \cA$ and $\xi, \eta \in H$ we have,
 	\[
 	     \langle  x \cdot (a \notimes \xi) \cdot y, b \notimes \eta   \rangle =
 	     \langle  \Psi^{b^\ast, a}(x) \xi y, \eta \rangle. 	
 	\]
 \end{lem}
 \begin{proof}
 	Indeed, we find that,
 	 	\[
 	\begin{split}
 	\langle  x \cdot (a \notimes \xi) \cdot y, b \notimes \eta   \rangle = &
 	\langle    (xa \notimes \xi y - x \notimes a \xi y), b \notimes \eta \rangle \\
 	= &
 	\langle  (\Gamma(xa, b) - \Gamma(x, b)a) \xi y,  \eta \rangle \\
 	= &
 	\frac{1}{2} \langle  (
 	\Delta(b^\ast) xa + b^\ast \Delta(xa) - \Delta(b^\ast xa)
 	- \Delta(b^\ast )xa - b^\ast \Delta(x) a + \Delta(b^\ast x) a )\xi y ,  \eta \rangle	\\
 	= &
 	\langle  \Psi^{b^\ast, a}(x) \xi y, \eta \rangle.
 	\end{split}
 	\]
 	This concludes the proof.
 	\end{proof}

  \begin{rmk}\label{Rmk=Positive}
  	Let $H$ be an $A$-$A$-bimodule and let $\xi \in H$. Then the functional on $A \otimes_{{\rm max}} A^{\op}$ given by $x \otimes y^{\op} \mapsto \langle x \xi y , \xi \rangle$ is positive and therefore also its restriction to $A \otimes_{{\rm alg}} A^{\op} \rightarrow \mathbb{C}$  sends elements of the form $x^\ast x$ with  $x \in A \otimes_{{\rm alg}} A^{\op}$ to non-negative reals. Therefore if this map is $\minotimes$ continuous we obtain a positive map $A \minotimes A^{\op} \rightarrow \mathbb{C}$.
  \end{rmk}

	For any $n\in \N_{\geq 1}$, we say that a vector $\xi \in L^2(M)_{\nabla^{(n)}}$ is \textit{algebraic} if it is contained in a linear span of elements of $a_0\otimes_\nabla a_1\otimes_\nabla \cdots \otimes_\nabla a_n$ for some $a_0,\ldots,a_n\in \mathcal A$.

\begin{thm}\label{Thm=IGHS-GC}
 	Let $n \in \mathbb{N}_{\geq 1}$. Suppose that the quantum Markov semi-group $\Phi = (\Phi_t)_{t \geq 0}$ is gradient-$\cS_p$ for $p = 2 n$.

Then for any algebraic $\xi \in L_2(M)_{\nabla^{(n)}}$, $\overline{A\xi A}$ is contained in the coarse bimodule $L_2(M) \otimes L_2(M)$ as $A$-$A$ bimodules. 	
 	In particular, $L_2(M)_{\nabla^{(n)}}$ is contained in an infinite multiple of the coarse bimodule $L_2(M) \otimes L_2(M)$ as $A$-$A$ bimodules and the $A$-$A$ bimodule actions on $L_2(M)_{\nabla^{(n)}}$ extend to normal $M$-$M$ bimodule actions via this embedding.

 \end{thm}
 \begin{proof}
 	Take $a_0, a_1, \ldots,  a_n, b_0, b_1, \ldots, b_n \in \cA$.
 	Consider vectors $\alpha := a_0 \otimes \ldots \otimes a_{n-1} \otimes a_n \Omega_\tau, \beta := b_0 \otimes \ldots \otimes b_{n-1} \otimes b_n \Omega_\tau  \in L_2(M)_{\nabla^{(n)}}$.
For  $x, y \in \cA$ we get from Lemma \ref{Lem=GradientComp} that,
 	\[
 	\begin{split}
 	\langle x \cdot \alpha \cdot y, \beta \rangle
 	= & \langle x \cdot (  a_0 \notimes \ldots \notimes a_{n-1} \notimes a_n \Omega_\tau ) \cdot y,  b_0 \notimes \ldots \notimes b_{n-1} \notimes b_n \Omega_\tau \rangle  \\
 = &
 	\langle
 	\Psi^{b_{n-1}^\ast, a_{n-1}} \circ   \ldots \circ  \Psi^{b_1^\ast, a_1} \circ  \Psi^{b_0^\ast, a_0}(x)  a_n \Omega_\tau y, b_n \Omega_\tau  \rangle \\
 	= &  \langle b_n^{\ast}
 	\Psi^{b_{n-1}^\ast, a_{n-1}} \circ   \ldots \circ  \Psi^{b_1^\ast, a_1} \circ  \Psi^{b_0^\ast, a_0}(x)  a_n \Omega_\tau  ,  \Omega_\tau y^\ast \rangle \\
 	= & \langle b_n^{\ast}
 	\Psi^{b_{n-1}^\ast, a_{n-1}} \circ   \ldots \circ  \Psi^{b_1^\ast, a_1} \circ  \Psi^{b_0^\ast, a_0}(x)  a_n \Omega_\tau  ,  y^\ast \Omega_\tau \rangle,
 	\end{split}
 	\]
 	where the last equality uses the fact that $\Omega_{\tau} y^{\ast} = y^{\ast} \Omega_{\tau}$.
 	
 	Since each $\Psi^{b_j^\ast, a_j}$ is an element of $\cS_p$ with $p = 2n$ we find that
 	\[
  \Psi :=	\Psi^{b_{n-1}^\ast, a_{n-1}} \circ   \ldots \circ  \Psi^{b_1^\ast, a_1} \circ  \Psi^{b_0^\ast, a_0} \in \cS_2.
 	\]
 	Then also $x \Omega_\tau \mapsto  b_n^\ast \Psi(x) a_n\Omega_\tau$ is in $\cS_2$.
 	Therefore, there exists $\zeta_{\alpha, \beta}  \in  L_2(M) \otimes \overline{ L_2(M) }$   such that,
 	 	\begin{equation}\label{Eqn=MinBd}
 	 	 	\langle x \cdot \alpha \cdot y, \beta \rangle =
 	\langle x \Omega_\tau \otimes  \overline{ y^{\ast}} \overline{\Omega_{\tau}},  \zeta_{\alpha, \beta}   \rangle
 	 	\end{equation}
Since the map $M^{\op} \ni y^{\op} \mapsto \overline{y^{\ast}} \in \overline{M}$ is a $\ast$-isomorphism, we deduce that assigning \eqref{Eqn=MinBd}  to $x \otimes y^{\op} \in M \otimes_{{\rm alg}} M^{\op}$ is $\minotimes$-bounded and moreover normal.

 	
 	Now take   any algebraic $\alpha \in L^2(M)_{\nabla^{(n)}}$. The previous paragraph shows that the map $\rho: M \otimes M^{\op}: x \otimes y^{\op} \rightarrow \langle x \cdot \alpha \cdot y, \alpha \rangle$ extens to $M \barotimes M^{\op}$. Moreover, $\rho$ is positive on $\cA \otimes_{{\rm alg}} \cA^{\op} \rightarrow \mathbb{C}$, see Remark \ref{Rmk=Positive}.   We claim that $\rho$ is also a positive map $M \barotimes M^{\op} \rightarrow \mathbb{C}$. Indeed, take $z \in M \barotimes M^{\op}$ positive and write $z = d^\ast d, d \in M \barotimes M^{\op}$. By Kaplansky's density theorem, let $d_j \in \cA \otimes \cA^{\op}$ be a bounded net converging strongly to $d$. Then $d_j^\ast d_j \rightarrow d^\ast d = z$ weakly and hence $\sigma$-weakly since these topologies coincide on the unit ball. Since $\rho$ is normal (i.e. $\sigma$-weakly continuous), we get $0 \leq \rho(d_j^\ast d_j) \rightarrow \rho(z)$.   Since $L_2(M) \otimes L_2(M)$ is the standard Hilbert space for $M \barotimes M^{\op}$ there exists $\zeta_{\alpha} \in L_2(M) \otimes L_2(M)$ such that $\langle x \cdot \alpha \cdot y, \alpha \rangle = \langle x \cdot \zeta_{\alpha} \cdot y, \zeta_\alpha \rangle$.   As algebraic vectors in $L^2(M)_{\nabla^{(n)}}$ form a dense subspace, we conclude by Lemma \ref{Lem=Containment}.
 \end{proof}

 \begin{cor}\label{Cor=GradientCoarseConsequence}
Suppose that the quantum Markov semi-group $(\Phi_t)_{t \geq 0}$ is gradient-$\cS_2$; then it is gradient coarse.
 \end{cor}

 \subsection{Construction of $n$-cocycles.}
 The following proposition allows to construct cocycles at the expense of changing the bimodule. We will remedy this change by using Schatten-$\cS_p$ properties below.

 \begin{thm}  \label{Thm=Gdelta}
 	Let $H$ be an $A$-$A$-bimodule and let $f\colon \cA^{\otimes (n-1)} \rightarrow H$ be a linear map. Set $Gf\colon \cA^{\otimes n} \rightarrow H_\nabla$ by,
 	\[
 	(G f)(a_1 \otimes \ldots \otimes a_{n}) := a_1 \notimes f(a_2 \otimes \ldots \otimes a_{n}) \in H_\nabla.
 	\]
 	Then $Gd + dG=0$, where $d$'s are the differentials of the appropriate cochain complexes. In particular, $G$ maps cocycles to cocycles and coboundaries to coboundaries, so it induces a map on the level of cohomology $G_{\ast}\colon H^{n}(\cA, H) \to H^{n+1}(\cA, H_{\nabla})$.
 \end{thm}
 \begin{proof}
 	 Let us first compute $G(df)$. We have
\begin{align*}
df(a_1\otimes \dots\otimes a_n) &= a_1 f(a_2\otimes \dots\otimes a_n) + \sum_{k=1}^{n-1} (-1)^{k} f(a_1\otimes \dots \otimes a_k a_{k+1}\otimes \dots \otimes a_n)  \\
&+ (-1)^n f(a_1\otimes\dots\otimes a_{n-1})a_n,
\end{align*}
hence
\begin{align*}
G(df)(a_0\otimes\dots\otimes a_n) &= a_0\otimes_{\nabla} a_1 f(a_2\otimes \dots\otimes a_n) + \sum_{k=1}^{n-1} (-1)^k a_0\otimes_{\nabla} f(a_1\otimes\dots\otimes a_k a_{k+1}\otimes\dots\otimes a_n) \\
&+(-1)^n a_0\otimes_{\nabla} f(a_1\otimes\dots\otimes a_{n-1})a_n.
\end{align*}
On the other hand:
\begin{align*}
-d(Gf)(a_0\otimes\dots\otimes a_n)&=-a_0(a_1\otimes_{\nabla} f(a_2\otimes\dots\otimes a_n))+a_0a_1\otimes_{\nabla} f(a_2\otimes\dots\otimes a_n) \\
&+ \sum_{k=1}^{n-1} (-1)^{k} a_{0}\otimes_{\nabla} f(a_1\otimes\dots\otimes a_k a_{k+1}\otimes\dots\otimes a_n) \\
&+ (-1)^n a_{0}\otimes_{\nabla} f(a_1\otimes\dots\otimes a_{n-1}) a_n.
\end{align*}
Note that
\[
-a_0(a_1\otimes_{\nabla} f(a_2\otimes\dots\otimes a_n)) = -a_0 a_1 \otimes_{\nabla} f(a_2\otimes\dots\otimes a_n) + a_0\otimes_{\nabla} a_1 f(a_2\otimes \dots\otimes a_n).
\]
It follows that $G(df) = -d(Gf)$, i.e. $dG+Gd=0$.
 \end{proof}

 \begin{cor}\label{Cor=Derivation}
 	The map,
 	\[
 	\partial_n: \cA^{\otimes n} \rightarrow L_2(M)_{\nabla^{(n)}}: a_1 \otimes \ldots \otimes a_n \mapsto a_1 \notimes \ldots \notimes a_n \notimes \Omega_\tau.
 	\]
 	defines an $n$-cocycle, i.e. an element in $C^n(\cA, L_2(M)_{\nabla^{(n)}})$.
 \end{cor}
\begin{proof}
For $n=1$ we get $\partial_1(ab) = ab \notimes \Omega_\tau = a \cdot (b \notimes \Omega_\tau) +
 (a \notimes \Omega_\tau) \cdot b = a \partial_1(b) + \partial_1(a) b$ and so the corollary follows. For higher $n$ use Theorem \ref{Thm=Gdelta} inductively.
\end{proof}

\subsection{The Haagerup averaging technique}  We will show that under algebraic conditions gradient-$\cS_p$ properties behave well under Haagerup averaging techniques. We believe this result is of independent interest, though we will not use it in the subsequent sections.

 Let $\Lambda$ be a discrete group with group algebra $\mathbb{C}[\Lambda]$ and von Neumann algebra $(\mathcal{L}(\Lambda), \tau)$. We will identify $\Lambda$ with its image under the left regular representation. So for $\gamma \in \Lambda$ we will also write $\gamma$ for the associated unitary in $\mathcal{L}(\Lambda)$.  Let $(\Phi_t)_{t \geq 0}$ be a quantum Markov semi-group on $\mathcal{L}(\Lambda)$. Consider the map,
\begin{equation}\label{Eqn=Average}
\Psi_t = \mathbb{E} \circ  (\Phi_t \otimes \id) \circ \delta_\Lambda
\end{equation}
where $\delta_\Lambda: \mathcal{L}(\Lambda) \rightarrow  \mathcal{L}(\Lambda) \barotimes \mathcal{L}(\Lambda)$ is the comultiplication $\delta_\Lambda: \gamma \mapsto \gamma \otimes \gamma$ and $\mathbb{E}$ is the conditional expectation of $\mathcal{L}(\Lambda) \barotimes \mathcal{L}(\Lambda)$ onto the image of $\delta_\Lambda$ (post-composed with the inverse of $\delta_{\Lambda}$ so that the range is contained in $\mathcal{L}(\Lambda)$). Then $\Psi_t$ is a completely positive Fourier multiplier with symbol $\varphi_t(\gamma) = \tau( \Phi_t(\gamma) \gamma^{-1}), \gamma \in \Gamma$.

\begin{lem}\label{Lem=Average}
	Let $\Psi_t: \cL(\Lambda) \rightarrow \cL(\Lambda), t \geq 0$ be unital completely positive Fourier multipliers (not necessarily a semi-group) and suppose that for every $\gamma \in \Lambda$ the limit $\Delta_0^\Psi(\gamma) :=  \lim_{t \searrow 0} \frac{1}{t} (\gamma -  \Psi_t(\gamma))$ exists. Then $\Delta_0^\Psi$ defines a preclosed (unbounded) operator on $\ell_2(\Gamma)$ with closure $\Delta^\Psi$. Moreover, $\exp(-t \Delta^\Psi)$ is a quantum Markov semi-group.
\end{lem}
\begin{proof}
	Since $\Delta_0^\Psi$ acts diagonally on $\ell_2(\Lambda)$ with respect to the basis $\gamma \in \Lambda$ it is preclosed. The remaining statements follow by following the constructions in \cite{JolissaintMartin} and \cite[Proposition 5.5]{CaspersSkalskiCMP}. Alternatively, there is the following short proof. Let $\varphi_t$ be the symbol of the Fourier multiplier $\Psi_t$ which is positive definite and $\varphi_t(e) = 1$. Consider the state $\mu_t$ on $\mathbb{C}[\Lambda]$ that maps $\gamma \in \Lambda$ to $\varphi_t(\gamma)$. Then $\Delta_0(\gamma) := \lim_{t \searrow 0} \frac{1}{t} ( \mu_t - \varepsilon )$ defines a generating functional on $\mathbb{C}[\Lambda]$ where $\varepsilon(\gamma) = 1, \gamma \in \Lambda$ is the counit (see \cite[Section 6.2]{DFSW} for generating functionals). Therefore $\nu_t := \exp(-t \Delta_0)$ defines a convolution semi-group of states on $\mathbb{C}[\Lambda]$. Since the Fourier multiplier $\Psi_t =   (\id \otimes \nu_t)\circ \delta_\Lambda$ is a trace preserving ucp map it extends from $\mathbb{C}[\Lambda]$ to a quantum Markov semi-group on $\mathcal{L}(\Lambda)$ with generator $\Delta^{\Psi}$.
\end{proof}

\begin{lem}
	 Let $(\Phi_t)_{t \geq 0}$ be a quantum Markov semi-group on $\mathcal{L}(\Gamma)$ with generator $\Delta$  and assume $\mathbb{C}[\Gamma]$ is in the domain of $\Delta$.
	Consider  $\Psi_t$ as in \eqref{Eqn=Average} and let  $\Delta^\Psi$ be as in Lemma \ref{Lem=Average}. We have,
	\[
	\Delta^\Psi(\gamma)   =  \frac{d}{dt}\vert_{t = 0} \langle \Phi_t(\gamma) , \gamma \rangle  \gamma = \langle \Delta(\gamma), \gamma \rangle  \gamma, \qquad \gamma \in \Gamma \subseteq \mathbb{C}[\Gamma].
	\]
\end{lem}
\begin{proof}
	This is clear from the definitions.
	\end{proof}

The following theorem shows that under a domain condition on $\Delta$ we may average a quantum Markov semi-group to a semi-group of Fourier multipliers while retaining the property of being gradient-$\cS_p$.

\begin{thm}
	 Let $(\Phi_t)_{t \geq 0}$ be a quantum Markov semi-group on $\mathcal{L}(\Gamma)$ with generator $\Delta$   and  assume $\mathbb{C}[\Gamma]$ is in the domain of $\Delta$. If $(\Phi_t)_{t \geq 0}$ is gradient-$\cS_p$ then so is $\exp(-t \Delta^\Psi)$.
\end{thm}
\begin{proof}
	Take $\mu_1, \mu_2 \in \Gamma$ fixed. Since $(\Phi_t)_{t \geq 0}$ is gradient-$\cS_p$ we have that the assignment
	\[
	\Psi^{\mu_1, \mu_2}: \gamma \mapsto \Delta(\mu_1 \gamma \mu_2) +  \mu_1 \Delta( \gamma ) \mu_2 - \mu_1 \Delta(\gamma \mu_2) -  \Delta(\mu_1 \gamma  ) \mu_2  	
	\]
	is in $\cS_p$ as an operator on $\ell_2(\Gamma)$. Therefore, also
		\[
	\mu_1^{-1} \Psi^{\mu_1, \mu_2}(\:\cdot \:) \mu_2^{-1}:	\gamma \mapsto \mu_1^{-1} \Delta(\mu_1 \gamma \mu_2) \mu_2^{-1} +   \Delta( \gamma )   -   \Delta(\gamma \mu_2) \mu_2^{-1} -  \mu_1^{-1} \Delta(\mu_1 \gamma  )	
	\]
	is in $\cS_p$. Let $\mathbb{D}$  be the conditional expectation of $B(\ell_2(\Gamma))$ onto the diagonal operators in $B(\ell_2(\Gamma))$. Certainly     $\mathbb{D}(\mu_1^{-1} \Psi^{\mu_1, \mu_2}(\: \cdot \:) \mu_2^{-1})$ is in $\mathcal{S}_p$ and, further,
	\[
	  \mathbb{D}(\mu_1^{-1} \Psi^{\mu_1, \mu_2}(\:\cdot \:) \mu_2^{-1})( 	\gamma) = \mu_1^{-1} \Delta^{\Psi}(\mu_1 \gamma \mu_2) \mu_2^{-1} +   \Delta^{\Psi}( \gamma )   -   \Delta^{\Psi}(\gamma \mu_2) \mu_2^{-1} -  \mu_1^{-1} \Delta^{\Psi}(\mu_1 \gamma  ).	
	\]
	Multiplying again with $\mu_1$ and $\mu_2$ shows that the assignment
	\[
	\gamma \mapsto \Delta^\Psi(\mu_1 \gamma \mu_2) +  \mu_1 \Delta^\Psi( \gamma ) \mu_2 - \mu_1 \Delta^\Psi(\gamma \mu_2) -  \Delta^\Psi(\mu_1 \gamma  ) \mu_2
	\]
	is in $\cS_p$.  	
\end{proof}

  \section{Schatten $\cS_p$-estimates for the Ornstein-Uhlenbeck semi-group on $q$-Gaussians}\label{Sect=Orhnstein-Uhlenbeck}
  We show that $q$-Gaussians $\Gamma_q(\RH)$ with the Ornstein-Uhlenbeck semi-group are gradient-$\cS_p$ for a range of $q$ depending on $p$ and the dimension of the Hilbert space $\RH$.

  \subsection{$q$-Gaussian algebras}
  Let $\RH$ be a   finite dimensional real Hilbert space with complexification $H = \RH \otimes \mathbb{C}$. We denote by $I$ the complex conjugation on $H$ and we extend it to $H^{\otimes n}$ by $I(v_1\otimes\dots\otimes v_n):= Iv_n\otimes\dots\otimes Iv_1$. Let $q \in (-1, 1)$ and consider the symmetrization operator on $H^{\otimes n}$ given by
  \[
  P_q^n(\xi_1 \otimes \ldots \otimes \xi_n) = \sum_{\sigma \in S_n}  q^{i(\sigma)} \xi_{\sigma(1)} \otimes \ldots \otimes \xi_{\sigma(n)},
  \]
  where $i(\sigma)$ is the number of inversions in $\sigma$.
  Recall that an inversion is a pair $(a,b) \in \{ 1, \ldots, n \}$ with $a < b$ and $\sigma(a) > \sigma(b)$.   For $q \in (-1, 1)$ we set new inner products on $H^{\otimes n}$ by
  \[
  \langle \xi, \eta \rangle_q =   \langle P_q^n \xi, \eta \rangle,
  \]
  and we refer to the Hilbert space with this $q$-deformed inner product as $H_q^{\otimes n}$. Also set the $q$-Fock space $F_q(H) = \mathbb{C} \Omega \oplus (\oplus_{n=1}^\infty H_q^{\otimes n})$,     with $\Omega$ a unit vector called the vacuum vector. Note that the conjugation $I$ extends to an antiunitary operator on $F_q(H)$; a typical permutation does not preserve the $q$-deformed inner product, so it is a special feature of the permutation reversing the order.
  We set the left creation operator for $\xi \in H$,
  \[
  l_q(\xi) \eta_1 \otimes \ldots \otimes \eta_n = \xi \otimes \eta_1 \otimes \ldots \otimes \eta_n,
  \]
  and the annihilation operator $l_q^\ast(\xi) = l_q(\xi)^\ast$.  These operators are bounded and extend to $F_q(H)$ for $q \in (-1, 1)$. We define the {\bf $q$-Gaussian algebra} as
  \[
  \Gamma_q := \Gamma_q(\RH) := \cA_q'', \qquad \cA_q := \ast{\rm -alg} \{ \: l_q(\xi) + l_q(\xi)^\ast \mid \xi \in \RH  \}.
  \]
  We let $\varphiO$ be the (tracial) vacuum state on $\Gamma_q(\RH)$ given by $\varphiO(x) = \langle x \Omega, \Omega \rangle$. $\Omega$ is a separating and cyclic vector for $\Gamma_q(\RH)$ and so   $\varphiO$ is faithful and $F_q(H)$ is the standard Hilbert space of   $\Gamma_q$. In this setting $\Omega$ is a tracial vector and further $\Gamma_q$ is a II$_1$-factor \cite{RicardCMP}.

 For vectors $\xi_1, \ldots, \xi_n \in H$ there exists a unique operator $W_q(\xi_1 \otimes \ldots \otimes \xi_n) \in \Gamma_q(\RH)$ such that,
 \[
 W_q(\xi_1 \otimes \ldots \otimes \xi_n) \Omega = \xi_1 \otimes \ldots \otimes \xi_n.
 \]
  We call these operators {\bf elementary Wick operators} or {\bf elementary Wick words}.
 The {\bf Ornstein-Uhlenbeck} quantum Markov semi-group is then defined by\footnote{It extends to the $q$-Gaussian algebra by \cite[Theorem 2.11]{BKS}.}
 \[
 \Phi_t ( W_q(\xi_1 \otimes \ldots \otimes \xi_n)) = e^{-t n}  W_q(\xi_1 \otimes \ldots \otimes \xi_n).
 \]
 It comes with a generator (the quantum Dirichlet form) given by the number operator
 \[
 \Delta (\xi_1 \otimes \ldots \otimes \xi_n) = n \xi_1 \otimes \ldots \otimes \xi_n.
 \]
 We determine the cohomology as follows.

 \begin{thm}\label{CartanEilenberg}
 	For any $\cA_q$-$\cA_q$-bimodule $H$ we have $H^n(\cA_q, H) = 0$ for $n \geq 2$.
 \end{thm}
\begin{proof}
 Let $e_1, \ldots, e_m$ be an orthonormal basis of $\RH$.
We claim that algebraically $\cA_q$ is isomorphic to the $\ast$-algebra of non-commutative polynomials in $m$ self-adjoint variables $X_1, \ldots X_m$ by the identification  $X_i := W_q(e_i)$.   Indeed, let $P$ be a non-commutative polynomial and suppose that $P(X_1, \ldots, X_m) = 0$.
Let  $Q(X_1, \ldots, X_n) = X_{i_1} \ldots X_{i_d}$ be a monomial occuring in $P$ of highest degree $d$.  Then, from Lemma \ref{Lem=WickProduct} below, we find $Q(X_1, \ldots, X_n) = W_q(X_{i_1} \otimes \ldots \otimes X_{i_d}) + A$ where $A$ is a linear combination of Wick operators $W_q(\xi_1 \otimes \ldots \otimes \xi_k)$ of tensors of order $k < d$. If $P(X_1, \ldots, X_m) = 0$ we must therefore have that $Q(X_1, \ldots, X_m) =0$. By induction we conclude that $P =0$.

Now in \cite[p. 192]{CartanEilenberg} it is shown that for the algebra of non-commutative polynomials,  the Hochschild cohomology with coefficients in any bimodule is 0. Alternatively, see  \cite[Eqn. (5.11)]{Kassel} for an explicit short exact projective resolution which also yields this result.
\end{proof}

 \subsection{Gradient-$\cS_p$ estimates}

Let $[n]$ be the set $\{ 1, \ldots, n \}$,
 with $n\in \N$. $[n]$ has its natural order. We will consider the operator $\Psi^{a,b}$ associated to the Ornstein-Uhlenbeck semigroup. We will need a formula for products of Wick words to be able to work with this operator. We will follow Effros-Popa \cite{EffrosPopa} (see also \cite{DeyaSchott}).
 \begin{dfn}
 Let $n_1,\dots,n_k$ be natural numbers. By $P^{\leqslant 2}(n_1\otimes\dots\otimes n_k)$ we denote the set of partitions of the set $[n]$, where $n=n_1+\dots+n_k$, with blocks of size at most two, such that there is no pairing inside the sets $\{1,\dots,n_1\}, \{n_1+1,\dots,n_1+n_2\},\dots, \{n-n_k+1,\dots, n\}$. For $\pi \in P^{\leqslant 2}(n_1\otimes\dots\otimes n_k)$ we will define the number of crossings $ cr (\pi)$. Let $\pi = (l_1,r_1),\dots, (l_{m}, r_{m}), s_1,\dots, s_l$ be the blocks of the partition, so $2m+l=n$; we will denote by $P(\pi)$ the set of pairs and by $S(\pi)$ the set of singletons. We will denote by $c(\pi)$ the usual crossing number, i.e. the number of pairs $(l_i, r_i)$ and $(l_j,r_j)$ such that $l_i<l_j<r_i<r_j$. The symbol $d(\pi)$ will denote the set of degenerate crossings, i.e. the number of triples $x<y<z$ such that $(x,z)=(l_i,r_i)$ for some $i$ and $y=s_j$ for some $j$, i.e. $y$ is not paired with anything. The crossing number $cr(\pi)$ is defined to be $cr(\pi)=c(\pi) + d(\pi)$.
 \end{dfn}
While cumbersome to formally define, the crossing number is easy to compute, provided that we have a graphical representation of the partition available:
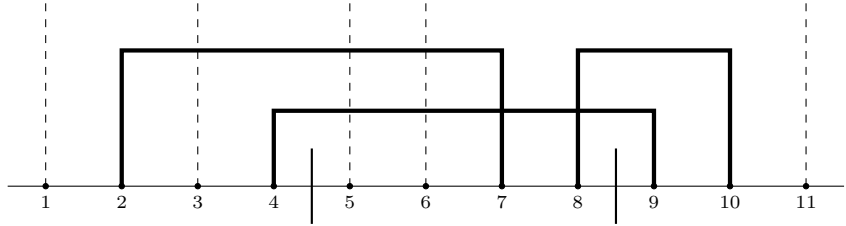
\begin{figure}[h]
\centering
\begin{tikzpicture}

\draw[thick] (4,-0.5)--(4,0.5);
\draw[thick] (8,-0.5)--(8,0.5);
\draw (0,0)--(11,0);
\foreach \x in {1,...,11}{
\filldraw[black] (-0.5+\x,0) circle (1pt) node[below] {{\tiny $\x$}};
}

\draw[ultra thick] (3.5,0)--(3.5,1)--(8.5,1)--(8.5,0);
\draw[ultra thick] (7.5,0)--(7.5,1.8)--(9.5,1.8)--(9.5,0);
\draw[ultra thick] (1.5,0)--(1.5,1.8)--(6.5,1.8)--(6.5,0);
\draw[dashed] (0.5,0)--(0.5,2.5);
\draw[dashed] (2.5,0)--(2.5,2.5);
\draw[dashed] (4.5,0)--(4.5,2.5);
\draw[dashed] (5.5,0)--(5.5,2.5);
\draw[dashed] (10.5,0)--(10.5,2.5);
\end{tikzpicture}
\end{figure}

For this partition $\pi$ there are two regular crossings (i.e. $c(\pi)=2$) and $5$ degenerate crossings, so $cr(\pi)=7$.

We are now ready to state the multiplication formula for Wick words (cf. \cite[Theorem 3.3]{EffrosPopa}).

\begin{lem}\label{Lem=WickProduct}
Let $\bm{\xi}=\xi_1\otimes \dots\otimes\xi_n \in H^{\otimes n}$, where $n=n_1+\dots+n_k$. Let us call $\bm{\xi}_1, \dots, \bm{\xi}_k$ the tensors that arise from this decomposition of $n$, i.e. $\bm{\xi}_1= \xi_1\otimes\dots\otimes \xi_{n_1}$, $\bm{\xi}_2= \xi_{n_1+1} \otimes \dots \otimes \xi_{n_1+n_2}$, etc.  We have the following formula	
	\begin{equation}\label{Eqn=WickProduct}
		W_q(\bm{\xi}_1)\dots W_q(\bm{\xi}_{k}) = \sum_{\pi \in P^{\leqslant 2}(n_1\otimes\dots\otimes n_k)}
	q^{cr(\pi)}\left(\prod_{(l,r)\in P(\pi)} \langle  I\xi_l, \xi_r \rangle\right)	W_q(\bm{\xi}_{S(\pi)}),
		\end{equation}
where $\bm{\xi}_{S(\pi)}:= \xi_{s_1}\otimes\dots\otimes \xi_{s_l}$ if $S(\pi)=\{s_1<s_2<\dots< s_l\}$.	
\end{lem}

This formula will be instrumental in the proof of the next proposition.

\begin{prop}\label{Prop=GradientQComp}
	Take operators in $\cA_q$ of the form
	\[
	a = W_q(\bm{\xi}) = W_q(\xi_1 \otimes \ldots \otimes \xi_n), b = W_q(\bm{\eta}) = W_q(\eta_1 \otimes \ldots \otimes \eta_k), x = W_q(\bm{\mu}) = W_q(\mu_1 \otimes \ldots \otimes \mu_m ),
	\]
	where each $\xi_i, \eta_i, \mu_i \in H$; define also $\bm{v}:=\bm{\xi}\otimes \bm{\mu}\otimes \bm{\eta}$. Then we have
	\begin{align}\label{Eqn=Laplaceformula}
	&\Delta(axb)- \Delta(ax)b + a \Delta(x) b - a\Delta(xb) \\
&= -2\sum_{\pi \in P^{\leqslant 2}(n\otimes m \otimes k)} \left|(l,r)\in P(\pi): l\in [n], r\in[k]\right|q^{cr(\pi)} 	\left(\prod_{(l,r)\in P(\pi)} \langle  Iv_l, v_r \rangle\right)	W_q(\bm{v}_{S(\pi)}). \notag
	\end{align}
	
\end{prop}
\begin{proof}
	 We determine the four summands in
	\begin{equation}\label{Eqn=DeltaFour}
	\Delta(axb) + a \Delta(x) b - \Delta(ax)b - a \Delta(xb).
	\end{equation}
	Note that the formula for a triple product from the Lemma \ref{Lem=WickProduct} can be obtained by applying the formula for double products twice. It means that we can compute $axb$ as either $(ax)b$ or $a(xb)$. The fact that the crossings that we compute on the way add up to the crossing number of the whole partition is not hard to check, but not immediate. For instance, if we compute $axb$ as $(ax)b$, then first we get a sum over partitions of $n\otimes m$ and then over partitions of $(n+m)\otimes k$, where you remove the nodes from $[n+m]$ that have already been paired; you have to check that the sum of the crossing numbers of these partitions is equal to the crossing number of the resulting partition of $n\otimes m \otimes k$.
	
Once we have this observation, it is not hard to arrive at the formula. Indeed, each of the four terms in \eqref{Eqn=DeltaFour} will feature a sum as in the formula for a triple product of Wick words, but with coefficients coming from the Laplacian. In case of $\Delta(axb)$ the coefficient is equal to $n+m+k - 2|P(\pi)|$. For $\Delta(ax)b$ the coefficient is equal to $n+m - 2|\{(l,r)\in P(\pi): l\in[n], r\in [m]\}|$. The term $a\Delta(x) b$ just yields $m$. The last one, $a\Delta(xb)$, produces $m+k - 2|\{(l,r)\in P(\pi): l\in[m], r\in[k]\}|$. If we compute $\Delta(axb) - \Delta(ax)b + a\Delta(x)b -a\Delta(xb)$ then the result is $-2|\{(l,r)\in P(\pi): l\in [n], r\in [k]\}|$. Indeed, the pairings between $[n]$ and $[m]$ are accounted for by $\Delta(ax)b$, the ones between $[m]$ and $[k]$ by $a\Delta(xb)$, and only pairings between $[n]$ and $[k]$ are left.
	\end{proof}

We now have a formula for $\Psi^{a,b}(x)$. In order to estimate the $\cS_p$-norm of this map, we will estimate separately the $\cS_p$-norms of $\Psi^{a,b}$ restricted to $H_q^{\otimes m}$ for each $m\geqslant 0$. Recall that on an $N$-dimensional Hilbert space we can estimate the $\cS_p$-norm by the operator norm, namely $\|T\|_{p} \leqslant N^{\frac{1}{p}} \|T\|$. Therefore we will be concerned only with the operator norm of $\Psi^{a,b}$ restricted to $H^{\otimes m}$.
\begin{prop}\label{Prop=Normestimate}
Let $\bm{\xi}:=\xi_1\otimes\dots\otimes\xi_n \in H^{\otimes n}$, $\bm{\eta}:=\eta_1\otimes\dots\otimes \eta_n \in H^{\otimes k}$ and consider $a:=W(\bm{\xi})$ and $b:=W(\bm{\eta})$. Then the norm of $\Psi^{a,b}$ restricted to $H_q^{\otimes m}$ can be estimated as follows
\begin{equation}
\|\Psi^{a,b}_{|H_q^{\otimes m}}\| \leqslant C(q, \bm{\xi}, \bm{\eta}, \dim(H) ) |q|^{m}.
\end{equation}
\end{prop}

In order to obtain useful estimates, we need to replace the formula \eqref{Eqn=WickProduct} by a new one, which takes into account the $q$-deformed inner products. We will only consider the case of $2$ or $3$ Wick words for simplicity. We will need operators $R_{n,k}^{\ast}\colon H^{\otimes n+k} \to H^{\otimes n+k}$ given by
\[
R_{n,k}^{\ast}(v_1\otimes \dots \otimes v_n) = \sum_{A\subset [n+k]: |A|=n} q^{i(A)} v_{A}\otimes v_{[n]\setminus A},
\]
where $i(A) = \sum_{l=1}^{n} (i_{l}-l)$ if $A=\{i_1<i_2<\dots < i_{n}\}$; one can interpret this number as the cost of moving elements of $A$ to the left of $[n+k]$ or, equivalently, cost of moving its complement to the right. Now we can rewrite the formula for the double product of Wick words.
\begin{lem}
For ($\bm{\mu}=\mu_1\otimes\dots\otimes \mu_m$ and $\bm{\eta}=\eta_1\otimes \dots \otimes \eta_k$) we have
\begin{equation}\label{Eqn=NewWickproduct}
W(\bm{\mu})W(\bm{\eta})\Omega = \sum_{j=0}^{m\wedge k} ({\rm Id}_{m-j}\otimes m_j \otimes {\rm Id}_{k-j})(R_{m-j,j}^{\ast}(\bm{\mu})\otimes R_{j,k-j}^{\ast}(\bm{\eta})),
\end{equation}
where $m_j: H_q^{\otimes j} \otimes H_q^{\otimes j} \to \mathbb{C}$ is the inner product pairing given by $m_j(v\otimes w) = \langle Iv, w\rangle$.
\end{lem}
\begin{proof}Note that in Lemma \ref{Lem=WickProduct} the crossings of partitions come from two sources; from pairing a subset of $[m]$ with a subset of $[k]$ (the degenerate crossings) and from crossings inside these subsets. If you fix the subsets and sum over all partitions that pair these two subsets, the crossings can be incorporated into the definition of the $q$-deformed inner product, and the numbers $i(A)$ appearing in the definition of $R_{n,k}^{\ast}$ count exactly the number of degenerate crossings.
\end{proof}

We can iterate this to obtain a formula for $W(\bm{\xi})W(\bm{\mu})W(\bm{\eta})$. But if we do not introduce a simpler notation, it will get very complicated. We will start from an element of a triple tensor product $H^{\otimes n}\otimes H^{\otimes m} \otimes H^{\otimes k}$. One thing that we will have to do is to apply inner product pairing between any two of them, and we need a notation for that. So we will use the notation $m_{j}^{ab}$ to denote the pairing of $H^{\otimes j}$ from between the $a$-th and $b$-th space, e.g. $m_{j}^{13}$ will first split $H^{\otimes n} \otimes H^{\otimes m} \otimes H^{\otimes k}$ into $H^{\otimes n-j}\otimes H^{\otimes j} \otimes H^{\otimes m} \otimes H^{\otimes j} \otimes H^{\otimes k-j}$ and then pair the two $H^{\otimes j}$ spaces. Another thing is that operators $R_{n,k}^{\ast}$ only split the tensor power into two factors, but we will have to do it again, so we need a notation for that as well. The defining property of operators $R_{n,k}^{\ast}$ is $P_q^{n+k}= (P_q^{n}\otimes P_q^{k}) R_{n,k}^{\ast}$ (cf. \cite[Lemma 5]{Ilona}). It motivates the following definition.
\begin{dfn}
We define the operator $R_{n,k,l}^{\ast}\colon H^{\otimes n+k+l}\to H^{\otimes n+k+l}$ to be the unique operator such that $P_q^{n+k+l} = (P_q^{n}\otimes P_q^{k}\otimes P_q^{l})R_{n,k,l}^{\ast}$.
\end{dfn}
Later on we will have to use boundedness of these operators. In order to achieve that goal we will need the following simple lemma.
\begin{lem}\label{Lem=Splitting}
We have
\[
R_{n,k,l}^{\ast} = (R_{n,k}^{\ast}\otimes {\rm Id}_{l}) R_{n+k,l}^{\ast} = ({\rm Id}_n\otimes R_{k,l}^{\ast})R_{n,k+l}^{\ast}.
\]
\end{lem}
\begin{proof}
The equality $P_q^{n+k} = (P_q^{n}\otimes P_q^{k})R_{n,k}^{\ast}$ means exactly that $R_{n,k}^{\ast}$ is the adjoint of the `identity' map $R_{n,k}\colon H_{q}^{\otimes n} \otimes H_q^{\otimes k} \to H_{q}^{\otimes n+k}$. By the same token $R_{n,k,l}^{\ast}$ is the adjoint of the identity map $R_{n,k,l}\colon H_q^{\otimes n}\otimes H_q^{\otimes k} \otimes H_q^{\otimes l}  \to H_q^{\otimes n+k+l}$, which can be obtained by composition $R_{n+k,l}(R_{n,k}\otimes {\rm Id}_{l})$ or $R_{n,k+l}({\rm Id}_{n} \otimes R_{k,l})$.

\end{proof}
We are ready to state a new formula for a triple product.
\begin{prop}\label{Prop=Tripleproduct}
Let $\bm{\xi}\in H^{\otimes n}$, $\mu\in H^{\otimes m}$, and $\bm{\eta} \in H^{\otimes k}$. Then we have
\begin{equation}\label{Eqn=Tripleproduct}
W(\bm{\xi})W(\bm{\mu})W(\bm{\eta})\Omega = \sum_{j,r,s} q^{r(m-j-s)} m_{r}^{13}m_{s}^{12}m_j^{23}(R_{n-r-s,r,s}^{\ast}(\bm{\xi})\otimes R_{s,m-s-j,j}^{\ast}(\bm{\mu}) \otimes R_{j,r,k-j-r}^{\ast}(\bm{\eta}))
\end{equation}
\end{prop}
\begin{proof}
The first step is to apply formula \eqref{Eqn=NewWickproduct} to $W(\bm{\mu})W(\bm{\eta})$. It means that we apply the operator $\sum_{j}m_{j}(R_{m-j,j}^{\ast}\otimes R_{j,k-j}^{\ast})$ to $\bm{\mu}\otimes \bm{\eta}$. We denote the result by $\bm{\mu}\otimes_{W}\bm{\eta}=\sum_{j}\bm{\mu}\otimes_{W}^{j} \bm{\eta}$. And then we repeat the procedure to obtain a formula for the triple product, but we have to be careful at this point. If we just used the formula for a double product, we would have to apply the operator $\sum_{p} m_{p}(R_{n-p,p}^{\ast}\otimes R_{p,m+k-2j}^{\ast})$ to $\bm{\xi}\otimes (\bm{\mu}\otimes_{W}^{j}\bm{\eta})$ and sum over $j$. Now the tensors $\bm{\mu}$ and $\bm{\eta}$ are mixed and we do not want to let that happen. So we will pair a subset of $[n]$ with a subset of $[m+k-2j]$ (with some elements removed by the first pairing) but we will split the latter into parts lying in $[m]$ and $[k]$; say they have cardinalities $s$ and $r$ (so $p=r+s$). The issue is that we have to move this whole subset to the left of $[m+k-2j]$, pair with the corresponding subset of $[n]$ and compute the inner product in $H_q^{\otimes s+r}$, which is not what we want; we prefer $H_q^{\otimes s}\otimes H_q^{\otimes r}$. Because of the formula $P_q^{s+r} = (P_q^{s}\otimes P_q^r)R^{\ast}_{s,r}$, we can achieve this goal by applying $({\rm Id}_{n-r-s} \otimes R^{\ast}_{s,r})$ to $R_{n-p,p}^{\ast}(\bm{\xi})$. Note that our pairing $m$ involves a second ingredient, the complex conjugation $I$, which reverses the order of the tensor (i.e. $I(v_1\otimes\dots\otimes v_n) = Iv_n\otimes\dots\otimes Iv_1$). So we should consider the operator $I P_q^{s+r}$ instead. To make things easier, we will denote by $I^{j}$ the conjugation $I$ acting on $H^{\otimes j}$. We have to compute $I^{s+r} P_q^{s+r} = I^{s+r}(P_q^{s}\otimes P_q^r) R_{s,r}^{\ast}$. Because $I^{s+r}$ reverses the order, we get $I^{s+r}(P_q^s\otimes P_q^r)R_{s,r}^{\ast} = (I^{r}P_q^{r} \otimes I^{s}P_q^{s})R_{r,s}^{\ast}$. Thus we can write $m_{s+r} = m_r m_s (R_{r,s}^{\ast}\otimes {\rm Id}_{s+r})$. It follows that we can replace the pairing $m_{s+r}$ by $m_r m_s$ at the cost of applying ${\rm Id}_{n-r-s}\otimes R_{r,s}^{\ast}$ to $R_{n-p,p}^{\ast}(\bm{\xi})$; by Lemma \ref{Lem=Splitting} this  is the same as $R^{\ast}_{n-r-s,r,s}(\bm{\xi})$.

Look now at the part coming from $\bm{\mu}\otimes_{W}^{j} \bm{\eta}$; we have a tensor of rank $s$ coming from $\bm{\mu}$ and a tensor of rank $r$ coming from $\bm{\eta}$. We have to think about the cost of moving it to the left of $[m+k-2j]$. The $\bm{\mu}$ part we can simply move to the left and it is accounted for by applying the operator $R_{s,m-j-s}^{\ast}$. The $\bm{\eta}$ part we move to the left of $[k-j]$ using the operator $R_{r,k-j-r}^{\ast}$ and then we have to move it to be adjacent to the $\bm{\mu}$ part. There are $m-j-s$ nodes that we have to cross, and the subset has cardinality $r$, so the cost is equal to $r(m-j-s)$. It means that our formula can be written as
\[
\sum_{j,r,s} q^{r(m-j-s)} m_{r}^{13}m_{s}^{12} (R_{m-r-s,r,s}^{\ast}(\bm{\xi})\otimes (R_{s,m-j-s}^{\ast}\otimes R_{r,k-j-r}^{\ast})(\bm{\mu}\otimes_{W}^{j}\bm{\eta})).
\]
It is not yet exactly the formula \eqref{Eqn=Tripleproduct}, but it is very close to it. Recall that $\bm{\mu}\otimes_{W}^{j}\bm{\eta}$ is equal to $m_{j}(R_{m-j,j}^{\ast}(\bm{\mu})\otimes R^{\ast}_{j,k-j}(\bm{\eta}))$. So we arrive at
\[
(R_{s,m-j-s}^{\ast}\otimes R_{r,k-j-r}^{\ast})(m_{j}(R_{m-j,j}^{\ast}(\bm{\mu})\otimes R^{\ast}_{j,k-j}(\bm{\eta}))).
\]
Note that we can either apply the pairing immediately, just like in the formula above, but we can also first apply identities on the spaces that will be paired, and apply the pairing afterwards. Thus we get
\[
m_j((R_{s,m-j-s}^{\ast}\otimes {\rm Id}_j)R_{m-j,j}^{\ast}(\bm{\mu}) \otimes ({\rm Id}_j\otimes R_{r,k-j-r}^{\ast})R_{j,k-j}^{\ast}(\bm{\eta})).
\]
By Lemma \ref{Lem=Splitting}, we get $(R_{s,m-j-s}^{\ast}\otimes {\rm Id}_j)R_{m-j,j}^{\ast} = R_{s,m-j-s,j}^{\ast}$ and $({\rm Id}_j\otimes R_{r,k-j-r}^{\ast})R_{j,k-j}^{\ast}= R_{j,r,k-j-r}^{\ast}$. In order to finish the proof, it suffices to add the decoration $23$ to $m_j$ to make it $m_j^{23}$, which exactly matches the formula \eqref{Eqn=Tripleproduct}.
\end{proof}
\begin{rmk}
In order to pass to the formula for $\Psi^{a,b}(\bm{\mu})$ we just need to multiply each summand in \eqref{Eqn=Tripleproduct} by $-2r$.
\end{rmk}
In order to finish the proof of Proposition \ref{Prop=Normestimate}, we just need bounds for the inner product pairings and the operators $R_{n,k}^{\ast}$.
\begin{lem}
Let $H$ be a Hilbert space. The inner product pairing $m\colon \overline{H}\otimes H \to \mathbb{C}$ has norm  $\sqrt{ \dim(H)}$.
\end{lem}
\begin{proof}
Let $(e_{i})_{i\in I}$ be an orthonormal basis of $H$. Then the norm of $v:=\sum_{i\in I} \overline{e_i} \otimes e_{i}$ is equal to $\sqrt{|I|}$, while $m(v)=|I|$, so the norm is at least $\sqrt{|I|}$. On the other hand, a simple Cauchy-Schwarz estimate shows that it is not larger than that.
\end{proof}
\begin{lem}
Recall that the norm of $R_{n,k}^{\ast}\colon H^{\otimes n+k} \to H^{\otimes n+k}$ is bounded by $C(q)$. Then the norm $R_{n,k}^{\ast}: H_{q}^{\otimes n+k} \to H_q^{\otimes n} \otimes H_q^{\otimes k}$ is not greater than $\sqrt{C(q)}$.
\end{lem}
\begin{proof}
Because of the formula $P_q^{n+k} = (P_q^{n}\otimes P_q^{k})R_{n,k}^{\ast}$, we get $P_{q}^{n+k} \leqslant C(q) P_q^{n}\otimes P_q^{k}$. Indeed, if we have an equality $A=BT$ for positive operators $A$ and $B$ then $A^2 = B TT^{\ast}B \leqslant \|T\|^2 B^2$, hence $A\leqslant \|T\|B$, as square root is an operator monotone function. The majorisation $P_{q}^{n+k} \leqslant C(q) P_q^{n}\otimes P_q^{k}$ shows that $\|R_{n,k}\|\leqslant \sqrt{C(q)}$, therefore also $\|R_{n,k}^{\ast}\|\leqslant C(q)$.
\end{proof}
\begin{cor}
We have $\|R_{n,k,l}^{\ast}\|\leqslant C(q)$, where $R_{n,k,l}^{\ast}$ is viewed as an operator from $H_q^{\otimes n+k+l}$ to $H_q^{\otimes n}\otimes H_q^{\otimes k} \otimes H_q^{\otimes l}$.
\end{cor}
\begin{proof}
By Lemma \ref{Lem=Splitting} we can express $R_{n,k,l}^{\ast}$ as $(R_{n,k}^{\ast}\otimes {\rm Id}_{l}) R_{n+k,l}^{\ast}$, and then we can appeal to the lemma above.
\end{proof}
\begin{proof}[Proof of Proposition \ref{Prop=Normestimate}]
Let $\bm{\mu} \in H_q^{\otimes m}$ and consider $\Psi^{a,b}(\bm{\mu})$. According to Proposition \ref{Prop=Tripleproduct} we have
\[
\Psi^{a,b}(\bm{\mu}) = -2\sum_{j,r,s}rq^{r(m-j-s)} m_{r}^{13}m_s^{12}m_j^{23}(R_{n-r-s,r,s}^{\ast}(\bm{\xi})\otimes R_{s,m-s-j,j}^{\ast}(\bm{\mu}) \otimes R_{j,r,k-j-r}^{\ast}(\bm{\eta})).
\]
We may assume that $r\geqslant 1$ (otherwise the corresponding summand is equal to zero), so the (absolute value of the) factor $q^{r(m-j-s)}$ is bounded above by $|q|^m$, up to a constant depending on $q$, $n$, and $k$. As each of $j$, $r$, and $s$ is bounded by either $n$ or $k$, the range of summation is finite, so it suffices to bound uniformly the norm of each summand. If we treat $R_{n,k}^{\ast}$ as a map from $H_q^{\otimes n+k}$ to $H_q^{\otimes n}\otimes H_q^{\otimes k}$, then each summand in \eqref{Eqn=Tripleproduct} should be thought of as an element of the space $H_{q}^{\otimes n-r-s} \otimes H_q^{\otimes m-s-j} \otimes H_q^{\otimes k-j-r}$ and this norm is bounded by a constant depending on $q$ and $\dim(H)$ (coming from the inner product pairing). Because the norm of the identity map $R_{n-r-s,m-s-j, k-j-r}\colon H_{q}^{\otimes n-r-s} \otimes H_q^{\otimes m-s-j} \otimes H_q^{\otimes k-j-r} \to H_q^{\otimes m+n+k - 2(j+r+s)}$ has norm at most $C(q)$, the norm in $H_q^{\otimes m+n+k - 2(j+r+s)}$ is bounded uniformly as well, which ends the proof.
\end{proof}

\begin{thm}\label{Thm=QGaussianSp}
Suppose that   $\vert q \vert < \dim(H)^{-\frac{1}{p}}$. Then the Ornstein-Uhlenbeck quantum Markov semi-group	$(\Phi_t)_{t \geq 0}$ is gradient-$\cS_{p}$. If $\vert q \vert \leqslant  \dim(H)^{-\frac{1}{p}}$ then $(\Phi_t)_{t \geq 0}$ is immediately gradient-$\cS_{p}$.
\end{thm}
\begin{proof}
We we will estimate the $\cS_{p}$ norm of $\Psi^{a,b}$ by $\sum_{m=0}^{\infty} \|\Psi^{a,b}_{|H_q^{\otimes m}}\|_{p}$. Because $\dim(H_q^{\otimes m}) = \dim(H)^{m}$, we may estimate $\|\Psi^{a,b}_{|H^{\otimes m}}\|_{p}$ by $\dim(H)^{\frac{m}{p}}\|\Psi^{a,b}_{|H_q^{\otimes m}}\|$. The bound for this norm is provided by Proposition \ref{Prop=Normestimate}, so we get
\[
\|\Psi^{a,b}\|_{p} \leqslant C(q,\bm{\xi},\bm{\eta},\dim(H)) \sum_{m=0}^{\infty} \left(|q| \dim(H)^{\frac{1}{p}}\right)^{m}.
\] 	
This series is convergent if $|q|< \dim(H)^{-\frac{1}{p}}$.

If $|q| = \dim(H)^{-\frac{1}{p}}$, then we want to show that the maps $\Psi^{a,b}_{t}$ are in $\cS_p$ for any $t>0$. Running again the same computation, we get the estimate
\begin{align*}
\|\Psi_{t}^{a,b}\|_{p} &\leqslant C(q,\bm{\xi},\bm{\eta},\dim(H)) \sum_{m=0}^{\infty}e^{-tm} \left(|q| \dim(H)^{\frac{1}{p}}\right)^{m} \\
&= C(q,\bm{\xi},\bm{\eta},\dim(H)) \sum_{m=0}^{n} e^{-tm},
\end{align*}
and this series is convergent.
\end{proof}

\subsection{Consequences for Cartan rigidity}

We gather some Cartan rigidity properties; these were already obtained in an unpublished manuscript by Avsec \cite{Avsec}, where they were proved for any $-1 < q < 1$. Our cohomological properties in the next section are new however.  Recall that we assumed that $\dim(\RH) < \infty$.

We recall that a von Neumann algebra $M$ has the W$^\ast$CMAP (weak-$\ast$ complete metric approximation property) if there exists a net $\Phi_i: M \rightarrow M$ of completely contractive normal finite rank maps such that for every $x \in M$ we have $\Phi_i(x) \rightarrow x$ $\sigma$-weakly and $\limsup_i \Vert \Phi_i: M \rightarrow M \Vert_{cb} = 1$.

\begin{thm}
	For any  $\RH$ and  $\vert q \vert \leqslant  \dim(H)^{-\frac{1}{2}}$ the $q$-Gaussian algebra $\Gamma_q(\RH)$ is strongly solid.
\end{thm}
\begin{proof}
	By \cite{Avsec} we see that $\Gamma_q(\RH)$ has the W$^*$CMAP.  In Corollary \ref{Cor=Derivation}, we have constructed a closable derivation $\partial(=\partial_1): \cA_q \rightarrow L^2(\Gamma_q(\RH))_{\nabla}$.  This derivation is real by \cite[Lemma 3.10]{CaspersGradient} and proper as $\partial^\ast \overline{\partial} = \Delta$ and $\Delta$ has compact resolvent. Further, by Theorem \ref{Thm=QGaussianSp} we find that  $(\Phi_t)_{t \geq 0}$ is immediately gradient-$\cS_2$ and gradient $\cS_2$ if $\vert q\vert <  \dim(H)^{-\frac{1}{2}}$.

 In case  $\vert q \vert <  \dim(H)^{-\frac{1}{2}}$  from Theorem \ref{Thm=IGHS-GC} we see that the left and right $\cA_q$ actions on  $L^2(\Gamma_q(\RH))_{\nabla}$ extend to normal $\Gamma_q(\RH)$ actions. Moreover, the thus-obtained $\Gamma_q(\RH)$-$\Gamma_q(\RH)$  bimodule  $L^2(\Gamma_q(\RH))_{\nabla}$
is weakly contained in the coarse bimodule of $\Gamma_q(\RH)$. Then \cite[Corollary B]{OzawaPopaAJM} implies the result except that we need a reformulation of \cite[Corollary B]{OzawaPopaAJM}  in terms of derivations instead of cocycles; that version can be found in \cite[Appendix A]{CaspersGradient}   (even for stable normalizers).

For arbitrary $\vert q \vert \leq  \dim(H)^{-\frac{1}{2}}$ the proof follows in the same way using \cite[Proposition 4.3]{CaspersGradient} instead of  Theorem \ref{Thm=IGHS-GC}; one only needs to prove that the left and right $\cA_q$ actions on  $L^2(\Gamma_q(\RH))_{\nabla}$  extend to normal $\Gamma_q(\RH)$ actions. We do this in Proposition \ref{Prop=NormalExtension}.
\end{proof}

\begin{prop}\label{Prop=NormalExtension}
	The left and right $\mathcal{A}_q$-actions on  $L^2(\Gamma_q(\RH))_{\nabla}$  extend to commuting normal $\Gamma_q$-actions.
\end{prop}
\begin{proof}
The  proof  follows  \cite[Proposition 3.8]{CaspersGradient}.
We prove it for the left action, the proof for the right action is similar. We must show that the action is weakly continuous on the unit ball. Take elementary Wick operators $a = W_q(\bm{\xi}), b = W_q(\bm{\eta}), c = W_q(\bm{\mu}), d = W_q(\bm{\nu}) \in \cA_q$. For $\bm{\xi} = \xi_1 \otimes \ldots \otimes \xi_n$ we refer to $n$ as the length of $W_q(\bm{\xi})$. Consider,
\begin{equation}\label{Eqn=Inn}
\langle x \cdot a \notimes b\Omega, c \notimes d\Omega \rangle =
- \frac{1}{2} \langle d^\ast (\Delta(c^\ast x a) + c^\ast \Delta(x) a - \Delta(c^\ast x) a  - c^\ast \Delta(xa)) b\Omega,  \Omega \rangle
\end{equation}
Let $x,a,b,c,d$ have length $X, A, B, C, D$ respectively.
The operator $\Delta$ preserves the length of a Wick word. Then using  Lemma \ref{Lem=WickProduct} we see that each of the expressions $d^\ast \Delta(c^\ast x a) b$, $d^\ast c^\ast \Delta( x a) b$,  $d^\ast \Delta( c^\ast x  ) a b$ and   $d^\ast c^\ast \Delta( x )a b$ is a linear combination of Wick words of length at least $X - A - B - C - D$. Therefore if $X > A + B + C+D$ then \eqref{Eqn=Inn} is 0. We therefore see that the map
\begin{equation}\label{Eqn=Inn2}
\cA_q \rightarrow \mathbb{C}: x \mapsto \langle x \cdot a \notimes b\Omega, c \notimes d\Omega \rangle,
\end{equation}
factors through the normal projection $P_{\leq X}: \Gamma_q(\RH) \rightarrow \Gamma_q(\RH)$ onto Wick words of length $\leq X$. The range of this projection is finite dimensional and contained in $\cA_q$. Hence \eqref{Eqn=Inn2} extends to a normal map $\Gamma_q(\RH) \rightarrow \mathbb{C}$.

Now let $x \in  \Gamma_q(\RH)$ with $\Vert x \Vert \leq 1$ and let $x_j \in \Gamma_q(\RH)$ with $\Vert x_j \Vert \leq 1$ be a  net converging to $x$ weakly (Kaplansky's density theorem). Let $\xi, \eta \in  L^2(\Gamma_q(\RH))_{\nabla}$  and let $\xi_0, \eta_0$ be in the linear span of vectors of the form $a \otimes b \Omega$ with $a, b \in \cA_q$ as in the previous paragraph with $\Vert \xi - \xi_0 \Vert, \Vert \eta - \eta_0 \Vert < \varepsilon$.  Let $j_0$ be such that for $j > j_0$ we have $\vert \langle (x - x_j) \xi_0, \eta_0 \rangle \vert < \varepsilon$. Then,
\[
\begin{split}
\vert \langle (x - x_j) \xi, \eta \rangle \vert \leq & \vert \langle (x - x_j) \xi_0, \eta_0 \rangle \vert  + \Vert x - x_j \Vert \Vert \xi - \xi_0 \Vert \Vert \eta_0 \Vert + \Vert x - x_j \Vert \Vert \xi  \Vert \Vert \eta - \eta_0 \Vert\\
 \leq & \varepsilon + 2 \varepsilon       \Vert \eta_0 \Vert + 2 \varepsilon    \Vert \xi  \Vert.
\end{split}
\]
\end{proof}

 \section{Derivations and the Akemann-Ostrand property}\label{Sect=AO}

  In this section we show that quantum Markov semi-groups may be used to prove the Akemann-Ostrand property AO$^+$ of a von Neumann algebra. In particular we find a new range of $q$ for which $q$-Gaussians have AO$^+$. Let us first recall the definition of the Akemann-Ostrand property that is most suitable for us, see \cite{IsonoTransactions}.

\begin{dfn}
 	A finite von Neumann algebra $\cM$  has condition AO$^+$ if there exists a $\sigma$-weakly dense unital C$^\ast$-subalgebra $A \subseteq \cM$ such that:
 	\begin{enumerate}
 		\item $A$ is locally reflexive \cite[Section 9]{NateTaka};
 		\item There exists a ucp map $\theta: A \minotimes A^{\op} \rightarrow B(L_2(\cM))$ such that $\theta( a \otimes b^{\op }) - a b^{\op}$ is compact for all $a,b\in A$.
 	\end{enumerate}
\end{dfn}

Let $M$ be a finite von Neumann algebra with faithful normal trace $\tau$. Let $\cA$ be a $\ast$-subalgebra that is $\sigma$-weakly dense in $M$ with C$^\ast$-closure $A$. Let $\partial: \cA \rightarrow H$ be a closable derivation to an $\cA$-$\cA$ bimodule $H$. Assume that $\cA$ has some linear basis $e_i, i \in I$ with $I$ countable that is moreover orthonormal in $L_2(M)$. Define the linear map,
\begin{equation}\label{Eqn=S}
S: \cA \rightarrow H: e_i \mapsto  \frac{\partial(e_i)}{\Vert \partial(e_i) \Vert},
\end{equation}
  where we put any unit vector in $H$ as $S(e_i)$ if $\partial(e_i)=0$.  Then the map $S$ is not necessarily bounded on $L_2(M)$ and therefore we state the following assumption.

\vspace{0.3cm}

\noindent {\bf Assumption 2.} Assume that $S$ is a bounded map $L_2(M) \rightarrow H$. Further, we assume that $S^*S$ is a Fredholm operator on $L^2(M)$.

\vspace{0.3cm}

Next, for every $x,y \in \cA$ we introduce a bounded linear map
\begin{equation}\label{Eqn=Txy}
T_{x,y}: L_2(M) \rightarrow H: \xi \mapsto x S(\xi) y - S(x \xi y).
\end{equation}
Let $S = V \vert S \vert$ be the polar decomposition of $S$ and
 note that $1-V^*V$ is a finite rank projection.
Denote by $K(L_2(M))$ the {\bf compact operators}.

\begin{prop}\label{Prop=Action}
Assume that $H$ is weakly contained in the coarse bimodule of $M$ as $\cA$-$\cA$ bimodules (namely, the assignment $x\otimes y^{\rm op} \in B(L_2(M)\otimes L_2(M))\mapsto xy^{\rm op}\in B(H)$ for $x,y\in \cA$ extends to a bounded $\ast$-homomorphism).
Suppose that Assumption 2 holds and that $T_{x,y}$ is compact for all $x,y \in \cA$.  Then there is a ucp map
\[
\theta: A \minotimes A^{\op} \rightarrow B(L_2(M)),
\]
such that $\theta (a \otimes b^{\op})  - a b^{\op}$ is compact for all $a,b \in A$.
\end{prop}
\begin{proof}
	Observe that one can write $T_{x,y} = xy^{\rm op} S - S xy^{\rm op}$ for all $x,y\in \cA$, where $x$ and $y^{\rm op}$ mean the corresponding left and right actions. Let $\Pi$ be the quotient map on $B(L_2(M))$ onto the Calkin algebra. Then by the assumption, for $x,y\in \cA$,
\begin{align*}
	0=\Pi(S^*T_{x,y})
	=\ &\Pi(S^*xy^{\rm op}S)  - \Pi(S^*S x y^{\rm op})\\
	=\ &\Pi(|S|)\Pi(V^*xy^{\rm op}V)\Pi(|S|)  - \Pi(|S|^2)\Pi(xy^{\rm op}).
\end{align*}
It holds that
\begin{equation}\label{Eqn=AO}
	\Pi(|S|^2)\Pi( xy^{\rm op})= \Pi(|S|)\Pi(V^*xy^{\rm op}V)\Pi(|S|).
\end{equation}
By taking the adjoint of this equation and by exchanging $x,y$ by $x^*,y^*$, we get that
	$$\Pi(|S|^2)\Pi(xy^{\rm op}) = \Pi(xy^{\rm op})\Pi(|S|^2), \quad \text{for all } x,y\in \cA,$$
so that $\Pi(|S|)$ commutes with $\Pi(xy^{\rm op})$  for all $x,y\in \cA$.
Since $\Pi(|S|)$ is invertible by assumption, by applying $\Pi(|S|)^{-1}$ to \eqref{Eqn=AO} from left and right, we get
	$$\Pi( xy^{\rm op})= \Pi(V^*xy^{\rm op}V) , \quad \text{for all }x,y\in \cA.$$
We conclude that $V^*xy^{\rm op}V-xy^{\rm op}$ is compact for all $x,y\in \cA$.

	Finally by the weak containment assumption, there is a bounded $\ast$-homomorphism
	$$\varphi\colon \cA\otimes \cA^{\rm op} \to \mathrm{C}^*\{xy^{\rm op} \in B(H)\mid x,y\in \cA\}; \ \varphi(x\otimes y^{\rm op})=xy^{\rm op}.$$
Then  if $P:=V^*V=1$,
the composition $\Ad(V^*)\circ \varphi$ works as a desired ucp map $\theta$.
 For the general case, $\Ad(V^*)\circ \varphi$ is a ucp map into $P B(L_2(M))P$. Since $1-P$ is compact, by using a fixed state $\phi$ on $A\otimes_{\rm min}A^{\rm op}$, we can consider $\phi(\cdot)(1-P) + \Ad(V^*)\circ \varphi$ as a desired ucp map.
\end{proof}

We now investigate sufficient conditions such that the assumptions of Proposition \ref{Prop=Action} are satisfied.

\subsection{Group algebras}
We first consider group von Neumann algebras. This case was (essentially) discussed in \cite[Section 15]{NateTaka}, but we include all proofs for reader's convenience.

	Let $\Gamma$ be a discrete group. Let  $\pi\colon \Gamma \to \mathcal{U}(H)$ be a unitary representation and $b\colon \Gamma \to H$  a 1-cocycle for $\pi$ in the sense that $b(xy)=b(x)+\pi_xb(y)$ for all $x,y\in \Gamma$. Consider left and right actions of $\Gamma$ on $H\otimes \ell^2(\Gamma)$ given by $\pi_x \otimes \lambda_x$ and $1\otimes \rho_y$ for all $x,y\in \Gamma$ respectively.
Then using the canonical basis $(\delta_x)_{x\in \Gamma}$ for $\ell^2(\Gamma)$, one can construct a closable derivation by
	$$ \partial\colon \C[\Gamma] \to H\otimes \ell^2(\Gamma); \ \partial(x)=b(x)\otimes e_x, \quad x\in \Gamma.$$
  As above, put $\cA:=\C[\Gamma]$ and define  $S\colon \cA \to H\otimes \ell^2(\Gamma)$ by $S e_x := \partial(x)\|\partial(x)\|^{-1}$ (and put $Se_x =\xi\otimes e_x$ for any fixed unit vector $\xi \in H$ if $\partial(x)=0$).
Recall that $b$ is \textbf{proper} if $\|b(g)\|\to0$ whenever $g\to \infty$.
In this setting, we prove the following.

 \begin{prop}
 	Suppose that $b$ is proper and that $\pi$ is weakly contained in the left regular representation. Then $\partial$ and $S$ satisfy all the assumptions in Proposition \ref{Prop=Action}.
 \end{prop}
 \begin{proof}
	Since $S$ is an isometry, Assumption 2 is trivially satisfied.
Since $\pi \prec \lambda$, it holds that the representation $\Gamma \times \Gamma \ni (x,y )\mapsto \pi_x\otimes \lambda_x\rho_y$ is weakly contained in the one of $\Gamma \times \Gamma \ni (x,y )\mapsto \lambda_x\otimes \lambda_x\rho_y$, which is in turn unitarily equivalent to $\Gamma \times \Gamma \ni (x,y )\mapsto \pi_x\otimes \rho_y$ via the unitary $e_x\otimes e_y \mapsto e_x \otimes e_{xy}$ on $\ell^2(\Gamma) \otimes \ell^2(\Gamma)$. We conclude that $H\otimes \ell^2(\Gamma)$ is weakly contained in the coarse bimodule as $\cA$-$\cA$ bimodules.

	It remains to show that $T_{x,y}$ is compact for all $x,y\in \mathcal A$. For this we can assume $x,y\in \Gamma$. Using Lemma \ref{Lem=OrderEstimate} below and using the properness of $b$, one has $\|T_{x,y} (e_g)\| \to 0$ as $g \to \infty$.
Observe that $T_{x,y} (e_g) \in  H\otimes \C e_{xgy}$, hence $T_{x,y}^*T_{x,y} \in \ell^\infty(\Gamma)$, so $T_{x,y}^* T_{x,y} \in c_0(\Gamma)$. We conclude that $T_{x,y}$ is compact.
 \end{proof}

\begin{lem}\label{Lem=OrderEstimate}
	For every $\gamma_1, \gamma_2 \in \Lambda$, there is a constant $C_{\gamma_1,\gamma_2}>0$ such that
	\[
	\Vert T_{\gamma_1, \gamma_2}(\mu) \Vert \leqslant C_{\gamma_1, \gamma_2} \Vert \partial(\mu) \Vert^{-1}.
	\]
\end{lem}
\begin{proof}
We have,
	\begin{equation}\label{Eqn=TxyEstimate}
	\begin{split}
	   T_{\gamma_1, \gamma_2}(\mu) = & \Vert \partial(\mu) \Vert^{-1} \gamma_1 \partial(\mu) \gamma_2 -   \Vert \partial(\gamma_1 \mu \gamma_2) \Vert^{-1}  \partial(\gamma_1 \mu \gamma_2)  \\
	   		=	&  ( \Vert \partial(\mu) \Vert^{-1}  -
	   		\Vert \partial(\gamma_1 \mu \gamma_2) \Vert^{-1}) \gamma_1 \partial(\mu) \gamma_2
	   		-   	
	   	 	\Vert \partial(\gamma_1 \mu \gamma_2) \Vert^{-1} ( \gamma_1  \partial(\mu) \gamma_2 +  \gamma_1    \mu  \partial(\gamma_2) ). \\
	\end{split}
	\end{equation}
We have that,
\[
	\Vert \partial(\gamma_1 \mu \gamma_2) \Vert - 	\Vert \partial( \mu ) \Vert \leqslant 	\Vert   \partial( \gamma_2) \Vert + \Vert   \partial( \gamma_1) \Vert,
\]
so
\[
\Vert \partial(\mu) \Vert^{-1}  - \Vert \partial(\gamma_1 \mu \gamma_2) \Vert^{-1} \leqslant \frac{\Vert   \partial( \gamma_2) \Vert + \Vert   \partial( \gamma_1) \Vert }{	\Vert \partial(\gamma_1 \mu \gamma_2) \Vert  	\Vert \partial( \mu ) \Vert},
\]
converges to 0 with order $O(\Vert \partial(\mu) \Vert^{-2})$ as  $\partial$ is proper. We conclude that  $\Vert T_{\gamma_1, \gamma_2}(\mu) \Vert$ converges to 0  with order $O(\Vert \partial(\mu) \Vert^{-1})$.
\end{proof}

  \subsection{Von Neumann algebras with filtration}
   We fix again $(M, \tau)$ a finite von Neumann algebra with $(\Phi_t)_{t \geq 0} = (\exp(-t \Delta))_{t \geq 0}$ a quantum Markov semi-group. Let
   \[
   \partial: \cA \rightarrow L_2(M)_\nabla: a \mapsto a \otimes \Omega_\tau
   \]
   be the derivation of Corollary \ref{Cor=Derivation}. We shall assume that $\Delta$ satisfies certain properties that are close to being a length function.

 	 \begin{dfn}
 	 	We say that $\Delta$ is {\bf filtered} if it has a compact resolvent and for every eigenvalue $\lambda$ of $\Delta$ there exists a (necessarily finite dimensional) subspace $\cA(\lambda) \subseteq \cA$ such that $\cA(\lambda)\Omega_\tau$ equals the eigenspace of $\Delta$ at eigenvalue $\lambda \geq 0$. Moreover, let $\lambda_n, n \in \mathbb{N}$ be an increasing enumeration of the eigenvalues of $\Delta$, and we assume the spaces   $\cA(\lambda_n), n \in \mathbb{N}$, to be filtered in the sense that
 	 	\begin{equation}\label{Eqn=Filter}
 	 	\bigoplus_{k=0}^{\infty} \cA(\lambda_{k})=\cA \quad \text{and} \quad \cA(\lambda_n) \cA(\lambda_m) \subseteq  \bigoplus_{k=0}^{m+n} \cA(\lambda_{k}) \quad \text{for all }n,m\in \N,
 	 	\end{equation}
where $\bigoplus$ means the algebraic direct sum.
	 \end{dfn}

 	 Suppose now that $\Delta$  has a compact resolvent with a complete set of eigenvectors $(e_i)_i$   in  $\mathcal A$ which is an orthonormal basis in $L^2(M)$   (for example if $\Delta$ is filtered this holds). Then since it is also a linear basis for $\cA$, we can define the map $S$ by $(e_i)_i$ and $\partial$ as in \eqref{Eqn=S}.
Observe that if $e_i \in \cA(\lambda_{n})$, then $\|\partial(e_i)\|^2 = \langle \Delta(e_i),e_i\rangle = \lambda_{n}$. Using this, it is easy to see that
\begin{equation}\label{Eqn=S2}
	S(a) = {\lambda_n^{-\frac{1}{2}}}{\partial(a)}, \quad \text{for all }a\in \cA(\lambda_n) \text{ and all eigenvalues }\lambda_n.
\end{equation}
In particular, the map $S$ does not depend on the choice of $(e_i)_i$.

 	 \begin{lem}\label{Lem=AO}
 	 Suppose that $\Delta$ has a compact resolvent with a complete set of eigenvectors in $\mathcal{A}$. Then the map $S$ in \eqref{Eqn=S2} is an isometry and hence satisfies Assumption 2.
 	 \end{lem}
  \begin{proof}
  	By definition, $S$ is given by eigenvectors $(e_i)_i$. Then we have,
  	\[
  	\begin{split}
  	\langle  S(e_i), S(e_j) \rangle  = &
  	\Vert \partial(e_i) \Vert^{-1} \Vert \partial(e_j) \Vert^{-1}   \langle \partial(e_i), \partial(e_j) \rangle =
  	\Vert \partial(e_i) \Vert^{-1} \Vert \partial(e_j) \Vert^{-1}   \langle \Delta^{\frac{1}{2}} e_i, \Delta^{\frac{1}{2}} e_j \rangle = \delta_{i,j}.
  	\end{split}
  	\] 	
  	So $S$ is an isometry.   	
  \end{proof}

  \begin{lem}\label{Lem=filtered}
  	Suppose that $\Delta$ is filtered, then in fact we have $\cA(\lambda_m)\cA(\lambda_n) \subset \bigoplus_{k= \vert m - n \vert}^{m+n} \cA(\lambda_k)$.
  \end{lem}
\begin{proof}
	As $\Delta(x^\ast) = \Delta(x)^\ast$ for $x \in \cA$, the spaces $\cA(\lambda)$ are self-adjoint. Now take $x \in \cA(\lambda_m)$, $y \in \cA(\lambda_n)$  and $z \in \cA(\lambda_k)$.  Assume that $m > n$  and that $k <  m - n$. By assumption we find that $y^\ast z \in \bigoplus_{l=0}^{n+k} \cA(\lambda_{l})$.  As $n+k < m$ this shows that $0 = \langle y^\ast z \Omega_\tau, x \Omega_\tau \rangle = \langle z \Omega_\tau, y x \Omega_\tau \rangle$. So $yx$ is orthogonal to $\cA(\lambda_k)$. The same holds for $xy$, which yields the lemma.
\end{proof}

  \begin{example}
  Let $\Lambda$ be a discrete finitely generated group and let $L: \Lambda \rightarrow \mathbb{N}$ be the length function given by the graph distance to the identity in the Cayley graph of $\Lambda$. Set $\Delta$ to be the closure of $\gamma \mapsto L(\gamma) \gamma$ as an unbounded operator on $\ell_2(\Lambda)$.   Suppose that $L$ is  conditionally positive definite so that
  \[
  Q(\xi) = \sum_{\gamma \in \Gamma}  L(\gamma)  \Vert \xi(\gamma ) \Vert^2 = \langle \Delta^{\frac{1}{2}} \xi, \Delta^{\frac{1}{2}} \xi\rangle
  \]
   is a quantum Dirichlet form (see \cite{CiprianiSauvageot}), i.e. the generator of a quantum Markov semi-group. Then $\Delta$ is easily seen to be filtered.
  \end{example}

\begin{rmk}\label{Rmk=FilteredExample}
	Let $\Gamma_q(\RH)$ be the $q$-Gaussian algebra for $-1 < q < 1$ with the Ornstein-Uhlenbeck semi-group $e^{-t \Delta}$, see Section \ref{Sect=Orhnstein-Uhlenbeck}. Then it follows from the Wick product formula of Lemma \ref{Lem=WickProduct} that $\Delta$ is filtered.
\end{rmk}

 	  \begin{dfn}
 	 We say that $\Delta$ has {\bf subexponential growth} if it has a complete set of eigenvalues $\lambda_0 < \lambda_1 < \lambda_2 < \ldots$ for which
 	 \begin{equation}\label{Eqn=Subexponential}
 	\lim_{n\rightarrow \infty} \frac{\lambda_{n+1}}{\lambda_n} = 1.
 	 \end{equation}
 	 \end{dfn}

  \begin{rmk}
  	Subexponential growth of a generator of a quantum Markov semi-group should be compared to the amenability results obtained in \cite{CiprianiSauvageot} and \cite[Appendix]{CaspersGradient}. These results show that often one cannot expect a growth on the eigenvalues that is more than linear and so in particular \eqref{Eqn=Subexponential} holds.
  	\end{rmk}

 \begin{thm}\label{Thm=Compact}
  	Let $\cM$ be a finite von Neumann algebra and let $(\Phi_t)_{t \geq 0} = (\exp(-t \Delta))_{t \geq 0}$ be a quantum Markov semi-group such that $\Delta$ is filtered with  subexponential growth. Let $S$ be as in \eqref{Eqn=S2}.
Then for every $x,y \in \cA $ the operator $T_{x,y}$ in \eqref{Eqn=Txy} is compact.   	
 \end{thm}
 \begin{proof}
 	Fix $x,y \in \bigoplus_{k = 0}^K \cA(\lambda_k)$ for some $K \in \mathbb{N}$. For each eigenvalue $\lambda$, let $P_{\lambda}$ be the orthogonal projection onto $\cA(\lambda) \Omega_\tau$ and we also regard it as a map $\cA \to \cA(\lambda)$.
Only in this proof, we will use the notation $\|a\|_\infty$ for the operator norm for $a\in \cA$.
Our first goal is to prove $\|T_{x,y}P_{\lambda_n}\| \to 0$ as $n \to \infty$. For this, we have only to prove that $\|T_{x,y} (a_n) \|  \to 0$ as $n\to \infty$, where $(a_n)_n$ is any sequence such that $a_n\in \cA(\lambda_n)$ and $\|a_n\Omega_\tau\|=1$ for all $n\in \N$.

	Take such a sequence $(a_n)_n$ and fix $n\in \N$. Then since $xa_n y\in \bigoplus_{k=-2K}^{2K} \cA(\lambda_{n + k})$  by Lemma \ref{Lem=filtered} (where $\cA(l)=0$ if $l<0$), one has $xa_ny = \sum_{k=-2K}^{2K}P_{\lambda_{n + k}}(xa_ny)$ (where $P_l=0$ if $l<0$), so that
	$$ S(xa_ny) = \sum_{k=-2K}^{2K}{\lambda_{n + k}^{-\frac{1}{2}}}\partial(P_{\lambda_{n + k}}(xa_ny)) .$$
Using this, we can write
	$$ T_{x,y}(a_n) =  x S(a_n) y - S(x a_n y) =  \lambda_{n}^{-\frac{1}{2}}{x \partial(a_n) y} -
 \sum_{k=-2K}^{2K}\lambda_{n + k}^{-\frac{1}{2}}\partial(P_{\lambda_{n + k}}(xa_ny) )$$
and further using $\partial(xa_ny) = \sum_{k=-2K}^{2K}\partial(P_{\lambda_{n + k}}(xa_ny))$, this translates into
 \begin{equation}\label{Eqn=VStep}
 T_{x,y}(a_n)
 =  \lambda_{n}^{-\frac{1}{2}} \left(  x \partial(a_n) y  -
 \partial(xa_ny)  \right)  + \sum_{k=-2K}^{2K}(\lambda_{n}^{-\frac{1}{2}}- \lambda_{n + k}^{-\frac{1}{2}})\partial(P_{\lambda_{n + k}}(xa_ny)) .
 \end{equation}
We will show that the first and the second term on the right hand side converges to $0$ as $n \to \infty$.

	We see the second term. Observe that for each $-2K \leqslant k\leqslant 2K$,
\begin{align*}
	\|\partial(P_{\lambda_{n + k}}(xa_ny))\|
	=  \|\Delta^{\frac{1}{2}}(P_{\lambda_{n + k}}(xa_ny))\|
	= \lambda_{n + k}^{\frac{1}{2}} \|P_{\lambda_{n + k}}(xa_ny)\|
	\leqslant \lambda_{n + k}^{\frac{1}{2}} \|x\|_\infty \|y\|_\infty.
\end{align*}
The subexponential growth condition then shows that for any such $k$,
\begin{align*}
	\|(\lambda_{n}^{-\frac{1}{2}}- \lambda_{n + k}^{-\frac{1}{2}})\partial(P_{\lambda_{n + k}}(xa_ny))\|
	\leqslant   \lambda_{n +k}^{\frac{1}{2}}|\lambda_{n}^{-\frac{1}{2}}- \lambda_{n + k}^{-\frac{1}{2}}|\|x\|_\infty\|y\|_\infty \to 0
\end{align*}
as $n\to \infty$. This finishes the case of the second term.

	We next see the first term. By using the Leibniz rule, our term is
 \begin{equation}\label{Eqn=RemainderTerm}
 \lambda_{n}^{-\frac{1}{2}}\left( x \partial(a_n) y - \partial(x a_n y) \right)
 =
 -\lambda_{n}^{-\frac{1}{2}}\left(   \partial(x)  a_n y + x a_n \partial(y) \right).
 \end{equation}
To estimate this term we firstly find
 \[
 \Vert \partial(x)  a_n y \Vert^2 = \Vert x  \notimes a_n y \Omega_\tau \Vert^2  \leqslant \Vert \Gamma(x, x) \Vert_\infty \Vert  a_n \Omega_\tau \Vert^2 \Vert y \Vert^2_\infty = \Vert \Gamma(x, x)  \Vert_\infty  \Vert y \Vert^2_\infty,
 \]
hence, as $n\to \infty$,
 \begin{equation}\label{Eqn=ConvergenceTwo}
 \lambda_{n}^{-\frac{1}{2}}  \Vert \partial(x)  a_n y \Vert \leqslant  \lambda_{n}^{-\frac{1}{2}} \Vert \Gamma(x, x)  \Vert^{\frac{1}{2}}_\infty  \Vert y \Vert_\infty \rightarrow 0.
 \end{equation}
 This shows that the first summand on the right hand side of  \eqref{Eqn=RemainderTerm} converges to 0. For the second summand, we first observe a couple of preliminary estimates.
Using the equation $a_ny = \sum_{k=-K}^K P_{\lambda_{n+k}}(a_ny)$ by Lemma \ref{Lem=filtered} (where $P_l=0$ if $l<0$), we have
\begin{align*}
	\Vert ( \Delta(a_n ) y   - \Delta(a_n y)) \Omega_\tau \Vert^2
	&= \left\Vert \sum_{k = - K}^K \lambda_{n } P_{\lambda_{n + k}} (a_n y) - \sum_{k = - K}^K \lambda_{n + k} P_{\lambda_{n + k}} (a_n y)   \right\Vert^2 \\
	&=  \sum_{k = - K}^K |\lambda_{n}- \lambda_{n+k}|^2\left\Vert P_{\lambda_{n + k}} (a_n y)\right\Vert^2 \\
	&\leqslant \sum_{k = - K}^K |\lambda_{n}- \lambda_{n+k}|^2\left\Vert y\right\Vert^2_\infty,
\end{align*}
so that by combining with the subexponential growth condition, as $n \rightarrow \infty$,
\begin{equation}\label{Eqn=Lim1}
	\lambda_{n}^{-1} \Vert ( \Delta(a_n ) y   - \Delta(a_n y)) \Omega_\tau \Vert \rightarrow 0.
\end{equation}
Secondly, as $\Delta$ is self adjoint,
\begin{equation}\label{Eqn=Lim2}
\lambda_{n}^{-1} \tau(  \Delta(y^\ast a_n^\ast a_n)  y) = \lambda_{n}^{-1} \tau(  y^\ast a_n^\ast a_n  \Delta(y))=
\sum_{k=0}^K \frac{\lambda_k}{\lambda_{n}} \tau(       y^\ast a_n^\ast a_n  P_{\lambda_k}(y)) \to 0,
\end{equation}
as $n\to \infty$, where we used the estimate for each summand as
\[
\begin{split}
 \frac{\lambda_k}{\lambda_{n}} |\tau(y^\ast a_n^\ast a_n P_{\lambda_k}(  y) )|
=\frac{\lambda_k}{\lambda_{n}} |\tau(     a_n P_{\lambda_k}(  y) y^\ast a_n^\ast ) |
\leqslant    \frac{\lambda_k}{\lambda_{n}} \Vert P_{\lambda_k}(  y) y^\ast \Vert_\infty,
\end{split}
\]
which converges to 0 as $n \rightarrow \infty$. Similarly,
\begin{equation}\label{Eqn=Lim3}
\lambda_{n}^{-1} \tau(  y^\ast \Delta(a_n^\ast a_n ) y )=\lambda_{n}^{-1} \tau(   a_n^\ast a_n  \Delta(y y^\ast) )
\rightarrow 0,
\end{equation}
as $n \rightarrow \infty$. Now we find that
\[
\begin{split}
& 2|\lambda_{n}^{-1} \tau( \Gamma(a_n, a_n y) y - y^\ast \Gamma(a_n, a_n) y  )| \\
\leqslant \  &
\lambda_{n}^{-1}  |\tau(  \Delta(  a_n y)^* a_n y   -  y^\ast \Delta(a_n)^* a_n y)| + \lambda_{n}^{-1}  |\tau(\Delta(y^\ast a_n^\ast a_n) y)| + \lambda_{n}^{-1} | \tau(y^\ast \Delta(a_n^\ast a_n) y )| \\
\leqslant \ &
\lambda_{n}^{-1}  \|(\Delta(a_ny )  -  \Delta(a_n)y )\Omega_\tau\| \|y\|_\infty + \lambda_{n}^{-1}  |\tau(\Delta(y^\ast a_n^\ast a_n) y)| + \lambda_{n}^{-1} | \tau(y^\ast \Delta(a_n^\ast a_n) y )|
\end{split}
\]
which converges to 0 as $n \rightarrow \infty$ by \eqref{Eqn=Lim1}, \eqref{Eqn=Lim2} and \eqref{Eqn=Lim3}.
Essentially the same estimates show that
\[
\begin{split}
& \lambda_{n}^{-1} \tau( \Gamma(a_n y, a_n y)  - y^\ast \Gamma(a_n y, a_n)   ) \rightarrow 0.
\end{split}
\]
We therefore get that, as $n\to \infty$,
\[
\begin{split}
\lambda_{n}^{-1} \Vert  a_n \partial(y) \Vert^2 = &  \lambda_{n}^{-1}\tau\left( \Gamma (a_n y, a_n y )  + y^\ast \Gamma(a_n, a_n) y - y^\ast \Gamma(a_n y, a_n) - \Gamma(a_n, a_n y) y
\right) \rightarrow 0.
\end{split}
\]
This shows that, as $n\to \infty$,
\begin{equation}\label{Eqn=ConvergenceThree}
\lambda_{n}^{-\frac{1}{2}}{\Vert x a_n \partial( y) \Vert}
\leqslant
\lambda_{n}^{-\frac{1}{2}}{ \Vert x \Vert_\infty \Vert  a_n \partial( y) \Vert} \rightarrow 0.
\end{equation} 	
In all, the convergences \eqref{Eqn=ConvergenceThree} and \eqref{Eqn=ConvergenceTwo} show that \eqref{Eqn=RemainderTerm} and hence \eqref{Eqn=VStep} goes to 0.
We conclude that $T_{x,y}(a_n) \to 0$ as $n\to \infty$ and therefore $\|T_{x,y} P_{\lambda_n}\| \to 0$ as $n\to \infty$.

	By Lemma \ref{Lem=filtered}, for any $n,m$ with $|n-m|>2K$,  $xa_n y$ and $xa_m y$ are orthogonal, so that
	$$ \langle xa_n\otimes_\nabla y, xa_m\otimes_\nabla y \rangle = \langle xa_n\otimes_\nabla y, xa_m y \otimes_\nabla 1 \rangle = \langle x\otimes_\nabla a_n y, xa_my\otimes_\nabla 1 \rangle = 0.$$
This implies $ \langle x \partial(a_n)y, x\partial(a_m)y\rangle = \langle x \partial(a_n)y, \partial(xa_my)\rangle=0$.
We obtain
	$$\langle T_{x,y}(a_n), T_{x,y} (a_m) \rangle=0, \quad \text{for all }n,m \text{ with }|n-m|>2K.$$
It turns out that $T_{x,y}^* T_{x,y} P_{\lambda_n} = \sum_{k=-2K}^{2K} P_{\lambda_{n+k}} T_{x,y}^* T_{x,y} P_{\lambda_n}$ for all $n\in \N$ (where $P_l=0$ if $l<0$).
By putting $T:=T_{x,y}^*T_{x,y}$, we see that
\begin{align*}
	T
	= \sum_{n\in \N} T P_{\lambda_n}
	= \sum_{n\in \N} \sum_{k=-2K}^{2K} P_{\lambda_{n+k}} T P_{\lambda_n}
	= \sum_{k=-2K}^{2K} \sum_{n\in \N}  P_{\lambda_{n+k}} T P_{\lambda_n},
\end{align*}
where the sum is in the strong topology. For each fixed $-2K \leqslant k \leqslant 2K $ and $m\in \N$, since $P_{\lambda_{n+k}}$ is orthogonal for different $n$, it holds that
\begin{align*}
	\left\|\sum_{m\leqslant n\in \N} P_{\lambda_{n+k}} T P_{\lambda_n}\right\|
	= \sup_{m\leqslant n\in \N}\|P_{\lambda_{n+k}} T P_{\lambda_n}\|
	\leqslant \sup_{m\leqslant n\in \N}\|T_{x,y}P_{\lambda_{n+k}}\| \|T_{x,y} P_{\lambda_n}\|,
\end{align*}
which converges to 0 as $m \to \infty$. Thus the sum $\sum_{n\in \N} P_{\lambda_{n+k}} T P_{\lambda_n}$ converges in the norm topology, hence it is a compact operator. We conclude that $T$ is compact, so $T_{x,y}$ is also compact.
\end{proof}

 We conclude as follows.

 \begin{thm}\label{Thm=AO}
 	Let $\cM$ be a finite von Neumann algebra and let $(\Phi_t)_{t \geq 0} = (\exp(-t \Delta))_{t \geq 0}$ be a quantum Markov semi-group that is gradient coarse and suppose that $\Delta$ is filtered with subexponential growth.   Assume further that $A$ as defined above is locally reflexive.  Then $\cM$ satisfies AO$^+$.
 \end{thm}
 \begin{proof}
 	Lemma \ref{Lem=AO} and Theorem \ref{Thm=Compact} show that the assumptions of Proposition \ref{Prop=Action} are satisfied. Therefore this proposition together with local reflexivity of $A$ implies that $M$ satisfies AO$^+$.
 \end{proof}

 The following corollary is then a consequence of \cite{IsonoTransactions} (see also  \cite{PopaVaesCrelle}).

 \begin{cor}
 	Suppose $M$ has the W$^*$CMAP. Then under the assumptions of Theorem \ref{Thm=AO}, $M$ is strongly solid.
 \end{cor}

	The following examples are covered by Theorem \ref{Thm=AO}.

\begin{cor}
	For $\vert q \vert \leqslant \dim(H)^{-\frac{1}{2}}$ the $q$-Gaussian algebra $\Gamma_q(\RH)$ satisfies Condition AO$^+$.
\end{cor}
\begin{proof}
	 In case $\vert q \vert < \dim(H)^{-\frac{1}{2}}$ we verify the conditions of  Theorem \ref{Thm=AO} as follows. From Theorem \ref{Thm=QGaussianSp} and Remark \ref{Rmk=FilteredExample} we see that a $\Gamma_q(\RH)$ admits a filtered quantum Markov semi-group that is gradient-$\cS_2$. From    Corollary \ref{Cor=GradientCoarseConsequence} this semi-group is gradient coarse and the conditions of Theorem \ref{Thm=AO} are verified. In case $\vert q \vert \leqslant  \dim(H)^{-\frac{1}{2}}$ we have that Theorem \ref{Thm=QGaussianSp} and Remark \ref{Rmk=FilteredExample} show that  $\Gamma_q(\RH)$ admits a filtered quantum Markov semi-group that is immediately gradient-$\cS_2$. Then from \cite[Proposition 4.3]{CaspersGradient} and Proposition \ref{Prop=NormalExtension} we see that the semi-group is gradient coarse. We conclude again by Theorem \ref{Thm=AO}.
\end{proof}

As mentioned before, Shlyakhtenko \cite{Shlyakhtenko} obtained the same result for $\vert q \vert < \sqrt{2} -1$ so that up to dimension 5 we find a new range. Other examples include the following.

\begin{example}
		Free group factors with the natural radial semi-group coming from the length function. Here conditon AO$^+$ is known, see \cite{NateTaka}.
\end{example}
	
	\begin{example}
		 Free orthogonal quantum groups $O_N^+$ (tracial case). In \cite{CaspersGradient} a gradient coarse quantum Markov semi-group was constructed which has the filter and subexponential growth property.  Together with local reflexivity (which follows from the CMAP of \cite{FreslonJFA} or \cite{CFY}) one obtains AO$^+$.  Here AO$^+$ was obtained already in \cite{VaesVergnioux} using  boundary actions.
\end{example}


\begin{thebibliography}{9}


\bibitem[Avs11]{Avsec}
  S. Avsec,
  \emph{Strong Solidity of the $q$-Gaussian Algebras for all $-1 < q < 1$},
  preprint, arXiv 1110.4918.


 \bibitem[BoSp91]{BozejkoSpeicher}
   M. Bozejko, R. Speicher,
   \emph{An example of a generalized Brownian motion},
    Comm. Math. Phys. {\bf 137} (1991), no. 3, 519--531.


  \bibitem[BKS97]{BKS}
M. Bozejko, B. K\"ummerer, R. Speicher,
\emph{$q$-Gaussian processes: non-commutative and classical aspects},
Comm. Math. Phys. {\bf 185} (1997), no. 1, 129--154.


\bibitem[CaEi56]{CartanEilenberg}
  H. Cartan, S. Eilenberg,
  \emph{Homological algebra},
  Princeton University Press, Princeton, N. J., 1956. xv+390 pp.

\bibitem[BrOz08]{NateTaka}
  N. Brown, N.  Ozawa,
  \emph{C$^\ast$-algebras and finite-dimensional approximations},
   Graduate Studies in Mathematics, 88. American Mathematical Society, Providence, RI, 2008. xvi+509 pp.


\bibitem[Cas18]{CaspersGradient}
M. Caspers,
\emph{Gradient forms and strong solidity of free quantum groups},
arXiv: 1802.01968.


\bibitem[CaSk15]{CaspersSkalskiCMP}
   M. Caspers, A. Skalski,
   \emph{The Haagerup approximation property for von Neumann algebras via quantum Markov semigroups and Dirichlet forms},
    Comm. Math. Phys. {\bf 336} (2015), no. 3, 1637--1664.


\bibitem[ChSi13]{ChifanSinclair}
   I. Chifan, T. Sinclair,
   \emph{On the structural theory of II$_1$ factors of negatively curved groups},
   Ann. Sci. \'Ec. Norm. Sup\'er. (4) {\bf 46} (2013), no. 1, 1--33 (2013).


\bibitem[Cip97]{Cipriani}
  F. Cipriani,
  \emph{Dirichlet forms and Markovian semigroups on standard forms of von Neumann algebras},
   J. Funct. Anal. {\bf 147}  (1997), 259–-300.


\bibitem[CFL00]{CiprianiFagnola}
   F. Cipriani, F. Fagnola, J.  Lindsay,
   \emph{Spectral analysis and Feller property for quantum Ornstein-Uhlenbeck semigroups},
    Comm. Math. Phys. {\bf 210} (2000), no. 1, 85--105.



  \bibitem[CiSa03]{CiprianiSauvageot}
F. Cipriani, J.-L. Sauvageot,
\emph{Derivations as square roots of Dirichlet forms},
J. Funct. Anal. {\bf 201} (2003), no. 1, 78--120.


  \bibitem[CiSa17]{CiprianiSauvageotAdvances}
F. Cipriani, J.-L. Sauvageot,
\emph{Amenability and subexponential spectral growth rate of Dirichlet forms on von Neumann algebras},
 Adv. Math. {\bf 322} (2017), 308--340.


\bibitem[CFY14]{CFY}
   K. de Commer, A. Freslon, M. Yamashita,
   \emph{CCAP for universal discrete quantum groups. With an appendix by Stefaan Vaes.},
   Comm. Math. Phys. {\bf 331} (2014), no. 2, 677--701.


\bibitem[CoSh05]{ConnesShlyakhtenko}
  A. Connes, D. Shlyakhtenko,
  \emph{$L^2$-homology for von Neumann algebras},
  J. Reine Angew. Math. {\bf 586} (2005), 125--168.

\bibitem[DaLi92]{DaviesLindsay}
 B. Davies, M. Lindsay,
 \emph{Noncommutative symmetric Markov semigroups},
  Math. Z. {\bf 210} (1992), no. 3, 379--411.

\bibitem[DFSW16]{DFSW}
   M. Daws, P. Fima, A. Skalski, S. White,
   \emph{The Haagerup property for locally compact quantum groups},
   J. Reine Angew. Math. {\bf 711} (2016), 189--229.

\bibitem[DeSch18]{DeyaSchott}
A.~Deya, R.~Schott,
	\emph{On multiplication in $q$-Wiener chaoses},
	Electron. Commun. Probab. {\bf 23} (2018), Paper no. 1, 16 pp.

   \bibitem[DyNi93]{NicaDykema}
    K. Dykema, A. Nica,
    \emph{On the Fock representation of the $q$-commutation relations},
     J. Reine Angew. Math. {\bf 440} (1993), 201--212.

\bibitem[EffPop03]{EffrosPopa}
  E.G. Effros, M. Popa,
  \emph{Feynman diagrams and Wick products associated with $q$-Fock space},
  Proc. Natl. Acad. Sci. USA {\bf 100} (2003), no. 15, 8629--8633.

\bibitem[Fre13]{FreslonJFA}
   A. Freslon,
   \emph{Examples of weakly amenable discrete quantum groups},
   J. Funct. Anal. {\bf 265} (2013), no. 9, 2164--2187.

\bibitem[GoLi95]{GoldsteinLindsay}
  S. Goldstein, M. Lindsay,
  \emph{KMS-symmetric Markov semigroups},
  Math. Z. 219(4), 591--608 (1995).


\bibitem[Jol04]{JolissaintMartin}
   P. Jolissaint, F. Martin,
  \emph{Alg\`ebres de von Neumann finies ayant la propri\'et\'e de Haagerup et semi-groupes L$^2$-compacts},
    Bull. Belg. Math. Soc. Simon Stevin {\bf 11} (2004), no. 1, 35--48.

\bibitem[JuMe12]{JungeMei}
  M. Junge, T. Mei,
  \emph{BMO spaces associated with semigroups of operators},
   Math. Ann. {\bf 352} (2012), no. 3, 691--743.



   \bibitem[Kas04]{Kassel}
   C. Kassel,
     \emph{Advanced School on Non-commutative Geometry},
   ICTP, Trieste, August 2004

\bibitem[Kro00]{Ilona}
I.~Kr\'{o}lak,
	\emph{Wick product for commutation relations connected with Yang-Baxter operators and new constructions of factors},
	Comm. Math. Phys. {\bf 210} (2000), no. 3, 685--701.


   \bibitem[Iso15]{IsonoTransactions}
   Y. Isono,
   \emph{Examples of factors which have no Cartan subalgebras},
   Trans. Amer. Math. Soc. {\bf 367} (2015), no. 11, 7917--7937.


\bibitem[Nou04]{Nou}
	A.~Nou,
	\emph{Non injectivity of the $q$-deformed von Neumann algebras},
	Math. Ann. {\bf 330} (2004), no. 1, 17--38.




\bibitem[OzPo10a]{OzawaPopaAnnals}
   N. Ozawa, S. Popa,
   \emph{On a class of II$_1$ factors with at most one Cartan subalgebra},
   Ann. of Math. (2) {\bf 172} (2010), no. 1, 713--749.


\bibitem[OzPo10b]{OzawaPopaAJM}
 N. Ozawa, S. Popa,
 \emph{On a class of II$_1$ factors with at most one Cartan subalgebra, II.}
  Amer. J. Math. {\bf 132} (2010), no. 3, 841--866.


\bibitem[OzPo04]{OzawaPopaPrime}
  N. Ozawa, S. Popa,
  \emph{Some prime factorization results for type II$_1$ factors},
  Invent. Math. {\bf 156} (2004), 223--234.



\bibitem[Pet09]{Peterson}
   J. Peterson,
   \emph{$L^2$-rigidity in von Neumann algebras},
    Invent. Math. {\bf 175} (2009), no. 2, 417--433.


\bibitem[Popa]{PopaIncrest}
  S. Popa,
  \emph{Notes on correspondences},
  INCREST.

  \bibitem[PoVa14]{PopaVaesCrelle}
   S. Popa, S. Vaes,
  \emph{Unique Cartan decomposition for II$_1$ factors arising from arbitrary actions of hyperbolic groups},
  J. Reine Angew. Math. {\bf 694} (2014), 215--239.


\bibitem[Ric05]{RicardCMP}
   E. Ricard,
   \emph{Factoriality of $q$-Gaussian von Neumann algebras},
   Comm. Math. Phys. {\bf 257} (2005), no. 3, 659--665.


\bibitem[Shl04]{Shlyakhtenko}
  D. Shlyakhtenko,
  \emph{Some estimates for non-microstates free entropy dimension with applications to q-semicircular families},
 Int. Math. Res. Not. {\bf 2004}, no. 51, 2757--2772.



\bibitem[Tho08]{ThomGafa}
  A. Thom,
  \emph{$L^2$-cohomology for von Neumann algebras},
  Geom. Funct. Anal. {\bf 18} (2008), no. 1, 251–270.



   \bibitem[VaVe07]{VaesVergnioux}
  S. Vaes, R. Vergnioux,
   \emph{The boundary of universal discrete quantum groups, exactness, and factoriality},
   Duke Math. J. {\bf 140} (2007), no. 1, 35--84.

   \bibitem[Voi96]{Voiculescu}
   D. Voiculescu,
   \emph{The analogues of entropy and of Fisher's information measure in free probability theory. III. The absence of Cartan subalgebras.}
   Geom. Funct. Anal. {\bf 6} (1996), no. 1, 172--199.




\end{thebibliography}
\end{document}